\documentclass[12pt,reqno]{amsart}
\usepackage{amssymb, amsfonts, amsbsy, latexsym, color}
\usepackage{epsfig}
\usepackage{subfigure}

\newtheorem{Thm}{Theorem}[section]

\newtheorem{corollary}[Thm]{Corollary}
\newtheorem{lem}[Thm]{Lemma}

\newtheorem{prop}[Thm]{Proposition}
\newtheorem{definition}[Thm]{Definition}

\newtheorem{theorem}[Thm]{Theorem}

\newcommand{\bitem}{\begin{itemize}}
\newcommand{\eitem}{\end{itemize}}
\newcommand{\benum}{\begin{enumerate}}
\newcommand{\eenum}{\end{enumerate}}
\newcommand{\beq}{\begin{equation}}
\newcommand{\eeq}{\end{equation}}
\newcommand{\ip}[2]{\langle#1,#2\rangle}

\newcommand{\absip}[2]{| \langle#1,#2\rangle |}

\newcommand{\spann}{\mbox{\rm span}}

  \newcommand{\R}{\mathbb{R}}

 \newcommand{\Z}{\mathbb{Z}}

\DeclareMathOperator*{\supp}{supp}

\def\ZZ{\mathbb{Z}}
\def\RR{\mathbb{R}}
\def\ZZ{\mathbb{Z}}

\def\SS{\mathbb{S}}

\newcommand{\bR}{{\mathbb R}}

\def\cB{{\mathcal{B}}}

\def\cD{{\mathcal{D}}}

\def\cSH{{\mathcal{S}\mathcal{H}}}

\newcommand{\gk}[1]{{\color{black}{#1}}}
\newcommand{\wq}[1]{{\color{black}{#1}}}

\begin{document}

\title{Irregular Shearlet Frames: Geometry and Approximation Properties}

\author[P. Kittipoom]{Pisamai Kittipoom}
\address{Faculty of Science, Prince of Songkla University, Hat Yai, Songkhla 90112, Thailand}
\email{pisamai.k@psu.ac.th}

\author[G. Kutyniok]{Gitta Kutyniok}
\address{Institute of Mathematics, University of Osnabr\"uck, 49069 Osnabr\"uck, Germany}
\email{kutyniok@math.uni-osnabrueck.de}

\author[W.-Q Lim]{Wang-Q Lim}
\address{Institute of Mathematics, University of Osnabr\"uck, 49069 Osnabr\"uck, Germany}
\email{wlim@math.uni-osnabrueck.de}

\thanks{P.K. acknowledges support from the Prince of Songkla
University Grant PSU 0802/2548. G.K. would like to thank Stephan
Dahlke, Christopher Heil, Demetrio Labate, Gabriele Steidl, and Gerd Teschke for
various discussions on related topics. G.K. and W.-Q L. acknowledge
support from DFG Grant SPP-1324, KU 1446/13-1.}

\begin{abstract}
Recently, shearlet systems were introduced as a means to derive efficient encoding methodologies
for anisotropic features in 2-dimensional data with a unified treatment of the continuum
and digital setting. However, only very few construction strategies for discrete shearlet
systems are known so far.

In this paper, we take a geometric approach to this problem. Utilizing the close connection
with group representations, we first introduce and analyze an upper and lower weighted shearlet
density based on the shearlet group. We then apply this geometric measure to provide necessary
conditions on the geometry of the sets of parameters for the associated shearlet systems to
form a frame for $L^2(\RR^2)$, either when using all possible generators or a large class
exhibiting some decay conditions. While introducing such a feasible class of shearlet generators,
we analyze approximation properties of the associated shearlet systems, which themselves lead to
interesting insights into homogeneous approximation abilities of shearlet frames. We also
present examples, such a oversampled shearlet systems and co-shearlet systems, to illustrate
the usefulness of our geometric approach to the construction of shearlet frames.
\end{abstract}



\maketitle

\section{Introduction}
\label{sec:intro}

A significant percentage of the data deluge we face today consists of 2D imaging data,
which issues the challenge to provide adapted efficient encoding methodologies. Typically,
this data is modeled in a continuum setting as functions in
$L^2(\RR^2)$. Recently, a novel directional representation system -- so-called
{\em shearlets} -- has emerged which provides a unified treatment of such continuum
models as well as digital models, allowing, for instance, a precise resolution of
wavefront sets, optimally sparse representations of cartoon-like images, and associated
fast decomposition algorithms (see \cite{GL07,DKS08,ELL08,KL09,Gro09a,Lim09,KL10} and
references therein).

Shearlet systems are, loosely speaking, systems generated by one single generator
with parabolic scaling, shearing, and translation operator applied to it. However,
at this moment, there exist only very few construction strategies for discrete shearlet
systems. Evidently, such strategies should in particular allow to monitor frame properties of
the generated system.

One construction approach consists in carefully choosing a feasible set
of (parabolic scaling, shearlet, translation)-parameters which leads to a shearlet frame, while
-- as a first step -- allowing all
choices of generators. A successful methodology in this
respect -- first introduced for Gabor systems, and then later applied to wavelet systems
(see the excellent survey paper \cite{Hei07} and the references therein) --
took the viewpoint of understanding the impact of the geometry of sets of parameters
to frame properties of the associated systems.

This will also be the viewpoint taken in this paper. It will lead us to different
necessary conditions on the geometry of the sets of parameters for the associated
discrete shearlet systems to form a frame for $L^2(\RR^2)$, either when using all possible
generators or a large class exhibiting some decay conditions. Interestingly,
along the way we will also analyze and prove homogeneous approximation properties of
shearlet systems.


\subsection{Weighted Shearlet Systems}
\label{subsec:weighted_shearlets}

Shearlet systems are designed to efficiently encode anisotropic features such as singularities
concentrated on lower dimensional embedded manifolds. To achieve optimal sparsity,
shearlets are scaled according to a parabolic scaling law encoded in the {\em parabolic scaling matrix}
$A_a$, $a>0$, and exhibit directionality by parameterizing slope encoded in the {\em shear matrix}
$S_s$, $s \in \bR$, defined by
\[
A_a =
\begin{pmatrix}
  a & 0\\ 0 & \sqrt{a}
\end{pmatrix}
\qquad\mbox{and}\qquad
S_s = \begin{pmatrix}
  1 & s\\ 0 & 1
\end{pmatrix},
\]
respectively. Hence, shearlet systems are based on three parameters: $a > 0$ being the {\em scale parameter}
measuring the resolution, $s \in \RR$ being the {\em shear parameter} measuring the directionality,
and $t \in \RR^2$ being the {\em translation parameter} measuring the position.
This parameter space $\RR^+ \times \RR \times \RR^2$ can be endowed with the group operation
\[
(a,s,t) \cdot (a',s',t') = (a'a,s'+s\sqrt{a'},t'+S_{s'}A_{a'}t),
\]
leading to the so-called {\em shearlet group} $\SS$. The expert reader will notice that this definition
deviates from the shearlet group first introduced in \cite{DKMSST08}. It will however be proven in
Subsection \ref{subsec:group} that these groups are isomorphic. The reasons for the deviation will
be discussed below. For now, let us say that this choice will allow us to measure the geometry of discretization
of the parameter space in a meaningful way.

As also first noticed in \cite{DKMSST08}, shearlet systems arise from a group representation of the
shearlet group. The shearlet group $\SS$ provides the unitary group representation -- we leave the
straightforward proof that this is indeed a group representation of $\SS$ to the reader --
\[
\sigma : \SS \to \mathcal{U}(L^2(\RR^2)), \quad (\sigma(a,s,t)\psi)(x) = a^{3/4} \psi(S_s A_a x-t),
\]
which will be the main building structure of continuous shearlet systems.
To obtain further flexibility, we extend the definition of the customarily utilized shearlet systems and
allow a weight function $w : \SS \to \RR^+$ to be incorporated in the definition of each single shearlet
function. For each $\psi \in L^2(\RR^2)$, we now define a {\em continuous weighted shearlet system} by
\begin{eqnarray*}
\cSH(\psi,w)
& = & \{w(a,s,t)^{1/2} \sigma(a,s,t)\psi : (a,s,t) \in \SS\}\\
& = & \{w(a,s,t)^{1/2} a^{3/4} \psi(S_s A_a \cdot\,-t) : (a,s,t) \in \SS\}.
\end{eqnarray*}
This definition coincides with the customarily defined continuous shearlet systems as defined, for instance,
in \cite{GKL06} -- except the additional flexibility provided by allowing a weight function. However,
the careful reader will notice that the different form -- due to utilizing a different group representation
-- requires a slightly different interpretation of the behavior of the parameters than in \cite{KL09},
but containing the same information. For some $\psi \in L^2(\RR^2)$, we further introduce the
{\em continuous weighted shearlet transform} given by
\beq \label{eq:CST}
\cSH_{\psi,w} : L^2(\RR^2) \to L^2(\SS, \tfrac{dadsdt}{a}), \quad \cSH_{\psi,w} f(a,s,t) = \ip{f}{w(a,s,t)^{1/2} \sigma(a,s,t)\psi}.
\eeq
If $w \equiv 1$, we sometimes also write $\cSH_{\psi}$.

Our main concern in this paper is to derive a deep geometric understanding of the discretization
process of the continuous weighted shearlet transform. To allow any possible discretization to be
investigated, we let $\Lambda$ be any discrete subset of $\SS$, and, for some function $\psi \in L^2(\RR^2)$
and some weight function $w : \Lambda \to \RR^+$, define the {\em weighted shearlet system} by
\[
\cSH(\psi,\Lambda,w) = \{w(a,s,t)^{1/2} a^{3/4} \psi(S_s A_a\cdot\,-t) : (a,s,t) \in \Lambda\}.
\]
(If $w \equiv 1$, we sometimes also write $\cSH(\psi,\Lambda)$.)
Such a shearlet system is typically referred to as an {\em irregular} (weighted) shearlet system
to separate it from the special case of {\em regular} (sometimes also called {\em classical}) shearlet
systems, which for $a  > 1$ and $b, c > 0$, are defined by
\beq \label{eq:classicalshearlets}
\{a^{3j/4} \psi(S_{bk} A_{a^j}\cdot\,-cm) : j, k \in \ZZ, m \in \ZZ^2\}.
\eeq
Observe that this system can be regarded as $\cSH(\psi,\Lambda,w)$ with $\Lambda = \{(a^{j}, bk, cm)
: j, k \in \ZZ,$ \\ $m \in \ZZ^2\}$ and $w \equiv 1$.

We should mention that the construction of continuous shearlet systems and their discretization
is paralleled by the construction of what became to be known as {\em shearlets on the cone} or
{\em cone-adapted shearlet systems}. These shearlet systems do not arise from a group representation,
but ensure a more equal treatment of the different directions. They are closely related to the
shearlet systems just defined, where the group operation is performed on four cones in which the
frequency domain is split into. See, for instance, the paper \cite{KKL10}, in which these systems
and the discretization procedure is precisely described. For our purposes though, the group structure
will play an essential role.


\subsection{Construction of Discrete Shearlets: State-of-the-Art and Fu\-ture Goals}

Shear\-lets were introduced in \cite{GKL06}, where only regular shearlet systems were considered.
For a more extensive, insightful study of the construction of regular shearlet systems with
`nice' frame properties, we refer to the recent paper \cite{Gro09b}.

However, to allow more flexibility in the choice of parameters, for instance, being able to
adapt it to specific applications and to choose it to ensure specific properties of the
associated shearlet system, as well as to analyze stability of the discretization of the
continuous shearlet transform requires a deep understanding of irregular shearlet systems.
The first notice of general irregular shearlet systems can be found in \cite{KL07}. In this
work some sufficient conditions for band-limited generators to lead to a shearlet frame with
control of the frame bounds were derived. This work was later extended by \cite{KKL10} to much
larger classes of shearlet generators, in particular, encompassing compactly supported generators.

However, these works include only a very particular class of discrete sets of parameters, those
which are in some sense `close' to the set of parameters associated with regular shearlet
systems. The main objectives of our analysis of irregular shearlet systems are therefore to
\bitem
\item[(1)] derive frame properties,
\item[(2)] consider all discrete subsets of the shearlet group as well as particular subsets,
\item[(3)] understand the stability of discretization of the continuous shearlet transform,
\item[(4)] analyze approximation properties.
\eitem

Certainly, we need some control on the properties of discrete subsets of $\SS$, and in this
regard we will take a geometric viewpoint. Our main guideline will be the interplay between
the geometry of the sets of parameters and properties such as frame or approximation properties
of the associated shearlet systems.


\subsection{A Geometric Analysis}

The first systems to be analyzed by a geometric viewpoint on the set of parameters were
irregular Gabor systems \cite{Lan93}. This started a series of papers with
different foci on this topic (see \cite{Hei07} for an excellent survey with a complete
reference list). The main idea was to use Beurling density for categorizing subsets of
$\RR^2$ according to their geometric structure and providing a relation to frame properties
of the associated Gabor systems. A very interesting generalization of these ideas was
lately studied in \cite{BCHL06}, which however did not encompass, for instance, wavelet systems.

From 2003 on, these ideas were also utilized in the study of irregular wavelet frames, for which a
different notion of density -- adapted to the affine group -- needed to be introduced
(see, for instance, \cite{HK03,SZ03,SZ04,Kut07b,HK07} and the book \cite{Kut07a}). However,
for analyzing irregular wave packet frames, additional complications occurred due to the
dimension of the parameter space compared to the domain of the generators when driving a
density approach; and ideas from the Hausdorff dimension had to be incorporated, see \cite{CKS06}.
A yet different path was taken in \cite{GKS08}, when introducing a notion of density for
arbitrary locally compact abelian groups which allowed to derive results paralleling
those from Landau \cite{Lan93} for these general groups.

With this paper, we follow the philosophy taken in \cite{Kut07a} for the analysis of
irregular wavelet frames. Based on the close connection of shearlet systems with a
particular unitary representation of the shearlet group $\SS$, we will introduce an
upper and lower shearlet density for providing meaningful measures for discrete subsets of $\SS$.
Armed with this measure, we derive necessary conditions in terms of finite upper and
positive lower shearlet density of a set of parameters for the associated shearlet
system -- either when using all possible generators or a large class exhibiting some
decay conditions -- to possess an upper and a lower frame bound, respectively (see
Theorem \ref{theo:main}). We will
also be able to distinguish different classes of sets of parameters, for instance, those
coming from regular shearlet systems we well as the here newly introduced oversampled
shearlet systems and co-shearlet systems, showing that co-shearlet systems can never
form a frame (see Section \ref{sec:examples}).

Summarizing, considering `only' a particular aspect of the geometry of a discrete subset of
$\SS$ leads to excluding conditions for the associated shearlet frames. However,
more is true. In fact, some of these conditions are paralleled by approximation
properties, which are worth considering them by themselves.


\subsection{Homogeneous Approximation}

Already in \cite{RS95} in the situation of Gabor systems and later (see, \cite{HK07})
for wavelet systems, the study of density properties was linked to specific approximation
properties of the accompanying systems, the so-called Homogeneous Approximation Property
(HAP). This link was exploited especially for deriving
necessary conditions on the existence of a lower frame bound. These ideas were later
extended to general coherent frames in the enlightening paper \cite{Gro08}, which however,
due to the generality of the setting, does not provide the full link to density considerations.

The common bracket in these studies is to analyze whether approximating a function by
a selection of elements of the system under consideration, with their parameters being
contained in a box in parameter space, is invariant under simultaneously translating
the function and the box; in this sense, whether this approximation is homogeneous.
The benefits of establishing this property are: (1) having control of approximation
properties of finitely many elements from a shearlet frame when shifting their set
of parameters by the group action, (2) obtaining insight into the possibility of
extending the elaborate framework from \cite{BCHL06} to the shearlet setting,
since the HAP is the heart of this theory, (3) understanding whether or not the
system exhibits a Nyquist phenomenon as Gabor frames do, and, (4) finally, to provide
the tools for deriving a necessary condition on the existence of a lower frame bound in
terms of density properties.

In this paper, we will show that indeed this property also holds for irregular shearlet frames
with generators satisfying some weak decay properties (see Theorem \ref{theo:hap}). For this,
we introduce a new class of shearlet generators $\cB_0$, whose associated continuous
shearlet transform belongs to a particular amalgam space, i.e., for which a particular
mixed (local-global) norm is finite. We anticipate that this class could also be useful
in other settings.


\subsection{Main Contributions}

Our main contribution of this paper is three-fold. Firstly, we derive a meaningful
geometric measure for analysis of arbitrary discrete subsets of the shearlet
group. Secondly, with Theorem \ref{theo:main}, we prove necessary conditions on the
geometry of the sets of parameters for the associated discrete shearlet systems to
form a frame for $L^2(\RR^2)$. And, thirdly, we analyze the ability of certain
irregular shearlet frames for homogeneously approximating functions (see
Theorem \ref{theo:hap}), thereby introducing a set of generators, $\mathcal{B}_0$,
which we expect to be useful for other tasks.


\subsection{Future Horizons}

In addition, this paper opens the door for the following future research directions.

\bitem
\item {\it Geometry of Parameters of Cone-Adapted Shearlets.}
The variant of shearlet systems we consider in this paper exhibits a useful
relation with group representations. Cone-adapted shearlet systems do not
possess this property, which makes density considerations for these especially
challenging. Nevertheless, success in this direction would
be highly beneficial for constructing this different class of shearlet
systems. We anticipate that the closeness of the two variants of shearlet
systems allows to carry over some of the main ideas from our analysis in this
paper.

\item {\it Density Analysis of Higher-Dimensional Wavelets.}
Higher-dimensional wavelet systems are associated with sets of parameters
$\{(A^j,bk): j \in \ZZ, k \in \ZZ^n\}$, where $A \in GL(n,\RR)$ and $b \in \RR^n$.
These systems are generated by the affine group $GL(n,\RR) \ltimes \RR^n$,
and the question arises whether geometric density conditions can also be
exploited here. However, in 2D this task is quite involved due to the
complicated inner structure of the group. Although associated with a slightly
different set of parameters, the analysis presented here could give some insights
at least for $n=2$.

\item {\it Homogeneous Approximation Property of Non-Band-Limited Shearlets.}
In this paper we show that the Homogeneous Approximation Property holds for a
shearlet frames being generated by band-limited generators satisfying some
weak spatial decay properties. This might provide some understanding of how
to extend this result to the situation of non-band-limited generators, which
are important for several practical applications.
\eitem


\subsection{Outline}

In Section \ref{sec:geometry}, we introduce the necessary group-theoretical framework
as well as the so-called weighted shearlet density, which will be our main means
for analyzing the geometry of sets of parameters of shearlet systems. Section \ref{sec:B0}
is then devoted to a particular set of generators with favorable decay conditions of
the associated continuous shearlet transform. Then, in Section \ref{sec:approximation},
we analyze approximation properties of shearlet systems, in particular, concerning
the Homogeneous Approximation Property. Section \ref{sec:necessary} then exploits
these results to prove the desired necessary conditions on the geometry of the sets of
parameters for the associated shearlet systems to form a frame for $L^2(\RR^2)$;
and applications of these to particular shearlet systems such as oversampled shearlet
systems and co-shearlet systems are the focus of Section \ref{sec:examples}. Since
several proofs are lengthy and highly technical, we outsourced those to Section \ref{sec:proofs}.


\section{Geometry of Shearlet Parameters}
\label{sec:geometry}

As already elaborated upon in the introduction, a geometric viewpoint will be our driving force.
The basic necessary framework to regard shearlet systems as being generated by a particular
locally compact group -- the shearlet group $\SS$ -- via a unitary group representation,
was already established in Subsection \ref{subsec:weighted_shearlets}. We now first discuss
this group in more detail and study the relation with the shearlet group introduced in
\cite{DKMSST08}. Based on this, we next introduce a means to `measure' the geometric
structure of discrete parameter sets in terms of denseness in the group $\SS$.

\subsection{Group Structure}
\label{subsec:group}

Let $\tilde{\SS}$ be the shearlet group introduced in \cite{DKMSST08}, which is defined as
the set $\RR^+ \times \RR \times \RR^2$ with group multiplication given by
\[
(a,s,t) \odot (a',s',t') = (aa',s+s'\sqrt{a},t+S_{s}A_{a}t').
\]
We have the following

\begin{lem}
The groups $\SS$ and $\tilde{\SS}$ are isomorphic via the group isomorphism
\[
\Phi : \SS \to \tilde{\SS}, \quad \Phi(a^{-1}, -\tfrac{s}{\sqrt{a}}, S_{-s/\sqrt{a}} A_{a^{-1}}t).
\]
\end{lem}

\begin{proof}
It is easily checked that $\SS$ and $\tilde{\SS}$ are groups (for $\tilde{\SS}$, see \cite{DKMSST08}).
Also, $\Phi$ is a bijective map. To prove the claim, let $(a,s,t), (a',s',t') \in \RR^+ \times \RR
\times \RR^2$ be arbitrary. Then
{\allowdisplaybreaks
\begin{eqnarray*}
\lefteqn{\Phi((a,s,t) \cdot (a',s',t'))}\\
& = & ((a'a)^{-1}, -\tfrac{s'+s\sqrt{a'}}{\sqrt{a'a}},S_{-(s'+s\sqrt{a'})/\sqrt{a'a}} A_{(a'a)^{-1}}(t'+S_{s'}A_{a'}t))\\
& = & (a^{-1} a'^{-1}, -\tfrac{s}{\sqrt{a}} -\tfrac{s'}{\sqrt{a'a}},S_{-s/\sqrt{a}} A_{a^{-1}}S_{-s'/\sqrt{a'}} A_{a'^{-1}}t'
+S_{-s/\sqrt{a}} A_{a'^{-1}}t)\\
& = & \Phi(a,s,t) \odot \Phi(a',s',t'),
\end{eqnarray*}
}
the lemma is proven.
\end{proof}

This now allows us to transfer results from \cite{DKMSST08} and \cite{DKST09} to the considered
situation. First, by exploiting the just defined isomorphism $\Phi$ and left-invariant Haar
measures of $\tilde{\SS}$ (see \cite{DKST09}), we can easily conclude that a left-invariant Haar
measure on $\SS$ is given by $\mu_\SS = \frac{da ds dt}{a}$.

For $\psi \in L^2(\RR^2)$ and $w : \SS \to \RR^+$, we further have the following correspondence
between $\cSH_{\psi,w}$ and the continuous shearlet transform from \cite{DKMSST08,DKST09},
\[
\widetilde{\cSH}_{\psi} : L^2(\RR^2) \to L^2(\tilde{\SS}),
\quad \widetilde{\cSH}_{\psi} f(a,s,t) = \ip{f}{a^{-3/4} \psi(A_a^{-1} S_s^{-1} (\cdot\,-t))}.
\]

\begin{lem} \label{lem:relation_SH_tildeSH}
For $\psi, f \in L^2(\RR^2)$,
\[
\widetilde{\cSH}_{\psi} f(a,s,t) = \cSH_{\psi} f (\Phi^{-1}(a,s,t)), \quad (a,s,t) \in \tilde{\SS}.
\]
\end{lem}

\begin{proof}
This is a direct consequence of the definitions of $\cSH$, $\widetilde{\cSH}$, and $\Phi$.
\end{proof}

This relation enables us to prove isometric properties of $\cSH$ with the help of \cite[Thm. 2.5]{DKST09}.

\begin{Thm}
Let $f \in L^2(\RR^2)$, and let $\psi \in L^2(\RR^2)$ be such that
\[
C_\psi^- = \int_{\mathbb{R}}\int_{-\infty}^0 \frac{|\hat{\psi}(\xi_{1},\xi_{2})|^{2}}{\xi_{1}^{2}}\, d\xi_{1}d\xi_{2} < \infty
\quad \mbox{and} \quad
C_\psi^+ = \int_{\mathbb{R}}\int_0^\infty \frac{|\hat{\psi}(\xi_{1},\xi_{2})|^{2}}{\xi_{1}^{2}}\, d\xi_{1}d\xi_{2} < \infty,
\]
i.e., $\psi$ is {\em admissible}. Then
\[
\int_\SS \absip{f}{\sigma(a,s,t)\psi}^2 \frac{da ds dt}{a}
= C_\psi^- \cdot \int_\RR \int_{-\infty}^0 |\hat{f}(\xi_1,\xi_2)|^2 d\xi_1 d\xi_2
+ C_\psi^+ \cdot \int_\RR \int_0^\infty |\hat{f}(\xi_1,\xi_2)|^2 d\xi_1 d\xi_2.
\]

In particular, if $C_\psi = C_\psi^-  = C_\psi^+$, then $\cSH_\psi$ is a $C_\psi$-multiple of an
isometry.
\end{Thm}

\begin{proof}
This follows immediately from \cite[Thm. 2.5]{DKST09} and Lemma \ref{lem:relation_SH_tildeSH}.
\end{proof}


\subsection{Weighted Shearlet Density}

We are now ready to introduce and derive some first basic properties of the geometric measure
for denseness of discrete weighted subsets of
$\SS$, which will be our main means to categorize and analyze different discretization
strategies of the continuous weighted shearlet transform in \eqref{eq:CST}. The upper
and lower weighted shearlet density will measure two different characteristics of a discrete subset
$\Lambda$ of $\SS$ endowed with weights $w : \Lambda \to \RR^+$:
\begin{itemize}
\item[(1)] the denseness as the average weighted number of points from $\Lambda$ lying in a `standard'
box moved in the 4-D space $\RR^+ \times \RR \times \RR^2$ by the group action, and
\item[(2)] the deviation of being `uniformly' spread in the 4-D space $\RR^+ \times
\RR \times \RR^2$ with respect to the group action.
\end{itemize}

\subsubsection{Definition}

The introduction of some system of `standard' boxes is our first concern. For the delicateness
of such and the significantly differing behavior of different choices in the affine group case,
we refer to \cite{Kut07a}. Let now $h > 0$ be the `size' of the boxes. Then we define
the box $Q_{h}$ to be the neighborhood of the unit element $(1,0,0)$ of $\SS$ given by
\[
Q_{h} = [e^{-h/2}, e^{h/2}) \times \left[-\tfrac{ h}{2}, \tfrac{h}{2}\right) \times \left[-\tfrac{h}{2}, \tfrac{h}{2}\right)^{2}.
\]
To allow measuring the density of a discrete set $\Lambda \subseteq \SS$ in the whole
4-D space $\RR^+ \times \RR \times \RR^2$, we define boxes $Q_h(x,y,z)$ centered at
arbitrary points $(x,y,z) \in \SS$ by left-translating $Q_h$ via the group action, i.e.,
\[
Q_{h}(x,y,z) = (x,y,z) \cdot Q_{h}, \quad h > 0.
\]
Notice that although we refer to these sets as `boxes', they are in fact versions
of the boxes $Q_h$ morphed by the group action.
By using the left-invariant Haar measure $\mu_{\mathbb{S}} = \frac{da ds dt}{a}$ derived in the
previous subsection, the volume of each $Q_h(x,y,z)$ can be computed to be
\[
\mu_{\mathbb{S}}(Q_{h}(x,y,z))
= \mu_{\mathbb{S}}(Q_{h})
= \int_{-\frac{h}{2}}^{\frac{h}{2}} \int_{-\frac{h}{2}}^{\frac{h}{2}} \int_{-\frac{h}{2}}^{\frac{h}{2}}
\int_{e^{-\frac{h}{2}}}^{e^{\frac{h}{2}}} \frac{1}{a}\, da\, ds\, dt_{1}\,dt_{2} =h^{4}.
\]

Since we are dealing with subsets $\Lambda$ of $\SS$ endowed with weights $w : \Lambda \to \RR^+$,
we need to decide how to incorporate the weights when counting the number of points of $\Lambda$
inside one of the just introduced boxes. For this, we define the weighted number of elements of
$\Lambda$ lying in a subset $K$ of $\mathbb{S}$ to be
\[
\#_{w}(K) = \sum_{(a,s,t) \in K} w(a,s,t).
\]

Now we collected all ingredients for introducing the notion of upper and lower shearlet density.

\begin{definition}\label{def:dens1}
Let $\Lambda$ be a discrete subset of $\mathbb{S}$, and let $w : \Lambda \to \RR^+$ be a weight
function. Then the \emph{upper weighted shearlet density} of $\Lambda$ with weights $w$ is
defined by
\[
\cD^{+}(\Lambda,w) = \limsup_{h \to \infty} \sup_{(x, y, z)\in \mathbb{S}}\frac{\#_{w} (\Lambda \cap Q_{h}(x, y, z))}{h^{4}},
\]
and the \emph{lower weighted shearlet density} of $\Lambda$ with weights $w$ is
\[
\cD^{-}(\Lambda,w) = \liminf_{h \to \infty} \inf_{(x, y, z)\in \mathbb{S}}\frac{\#_{w} (\Lambda \cap Q_{h}(x, y, z))}{h^{4}}.
\]
If $\cD^{+}(\Lambda,w) = \cD^{-}(\Lambda,w)$, we say that $\Lambda$ with weights $w$ has a
{\em uniform weighted shearlet density}.
\end{definition}

\subsubsection{Basic Properties}

In this subsection we will establish some basic properties which will be used extensively
through the paper. However, they are also illuminating by themselves as we will see. We
also wish to mention that the proofs of the two results use ideas from the proof of
\cite[Lem. 2.1, Prop. 2.2, and Prop. 2.3]{HK03}, however the action by the shearlet group
adds a significantly higher level of technical challenge to the analysis.

Weighted shearlet density measures the density in terms of the average weighted number of
points lying in boxes $Q_h(x,y,z)$. Intuitively, the upper weighted shearlet density of a
discrete weighted subset should be finite as long as the weighted points are sufficiently
separated, i.e., far apart, and the lower weighted shearlet density should be positive as
long as the weighted points are sufficiently dense, i.e., close together. Hence, for
deriving equivalent conditions to finite upper and positive lower weighted shearlet density,
the behavior of overlaps of these boxes will play an essential role.

To prepare this analysis, we first introduce the following two notions.

\begin{definition}
Let $X = \{x_{i}\}_{i \in I}$ be a sequence of elements in $\mathbb{S}$.
\begin{itemize}
\item[(i)] $X$ is called \emph{$Q_{h}$-dense in $\mathbb{S}$}, if
\[
\bigcup\limits_{i \in I} x_{i}\cdot Q_{h} = \mathbb{S}.
\]
\item[(ii)] $X$ is called \emph{separated}, if, for some $h > 0$, we have
\[
x_{i}\cdot Q_{h} \cap x_{j}\cdot Q_{h} = \emptyset, \quad i \ne j,
\]
and \emph{relatively separated}, if $X$ is a finite union of separated sets.
\end{itemize}
\end{definition}

Next, we study a particular `default' discrete subset -- related with the set of parameters
of a regular shearlet system --, with which we later on compare arbitrary discrete subsets.
We defer the lengthy, very technical proof to Subsection \ref{subsubsec:box}.

\begin{lem} \label{lem:covering}
Let $h > 0$ and $r \ge 1$, and define the discrete subset $X \subseteq \SS$ by
\[
X = \{(e^{jh}, he^{-h/4}k, he^{-h/2}m) : j, k \in \mathbb{Z}, m \in \mathbb{Z}^{2}\}.
\]
Then the following statements hold.
\begin{enumerate}
\item[(i)] $X$ is $Q_{h}$-dense in $\mathbb{S}$.
\item[(ii)] Any set $Q_{rh}(x, y, z)$ intersects at most
\[
N_{r}
:=(r+2)(r+1)^{3}\left[e^{h/2}+\frac{1}{r+1}\right]\left[e^{h}+\frac{1}{r+1}\right] \left[e^{3h/4}+\frac{1}{r+1}\right]
\]
elements in $X$. In particular, $X$ is relatively separated.
\item[(iii)] Any set $Q_{rh}(x, y, z)$ contains at least
\[
\tilde{N}_{r} :=r(r+1)^{3}e^{9h/4}
\]
elements in $X$.
\end{enumerate}
\end{lem}

This lemma now allows us to derive quite surprising equivalent conditions for
finite upper and positive lower weighted shearlet density. In fact, checking
these conditions can be reduced to checking the behavior of the weighted
discrete set with respect to the boxes $Q_h(x,y,z)$ for {\em one particular} $h > 0$.

\begin{prop} \label{prop:equivalent_D+finite_D-positive}
For $\Lambda \subset \mathbb{S}$ and $w: \Lambda \to \mathbb{R}^{+}$, the following conditions
are equivalent.
\begin{enumerate}
\item[(i)] $\cD^{+}(\Lambda,w) < \infty$.
\item[(ii)] There exists some $h > 0$ such that $\sup_{(x, y, z)\in \mathbb{S}}
\#_{w}(\Lambda \cap Q_{h}(x, y, z)) < \infty.$
\end{enumerate}
Also the following conditions are equivalent.
\begin{enumerate}
\item[(i')] $\cD^{-}(\Lambda,w) >0 $.
\item[(ii')] There exists some $h > 0$ such that $\inf_{(x, y, z)\in \mathbb{S}}
\#_{w}(\Lambda \cap Q_{h}(x, y, z)) > 0.$
\end{enumerate}
\end{prop}

\begin{proof}
(i) $\Rightarrow$ (ii). This follows immediately from the definition of $D^{+}(\Lambda,w)$.\\
(ii) $\Rightarrow$ (i). Let $h > 0$ be such that
\[
R := \sup_{(x, y, z)\in \mathbb{S}} \#_{w} (\Lambda \cap Q_{h}(x, y, z))< \infty.
\]
We first distinguish two cases:\\
In the case $0< t < h$, we have
\[
\#_{w} (\Lambda \cap Q_{t}(x, y, z)) \le \#_{w} (\Lambda \cap Q_{h}(x, y, z))\quad \mbox{for all }
(x,y,z) \in \mathbb{S},
\]
and hence
\[
\sup\limits_{(x, y, z)\in \mathbb{S}}\#_{w} (\Lambda \cap Q_{t}(x, y, z))< R.
\]
Now consider the case $t \ge h$, and assume $t = rh$, where $r \ge 1$. By Lemma \ref{lem:covering},
the box $Q_{rh}(x, y, z)$ is covered by a union of at most $N_{r}$ sets of the form $Q_{h}(e^{jh},
he^{-h/4}k, he^{-h/2}m)$. This implies
\[
\sup\limits_{(x, y, z)\in \mathbb{S}}\#_{w}(\Lambda \cap Q_{rh}(x, y, z))
\le N_{r}\cdot \sup\limits_{j,k \in \mathbb{Z},m\in \mathbb{Z}^{2}} \#_{w} (\Lambda \cap Q_{h}(e^{jh}, he^{-h/4}k,he^{-h/2}m))
\le N_{r}\,\cdot R.
\]
Summarizing both cases, we obtain
\begin{eqnarray*}
\cD^{+}(\Lambda)
& \le & \limsup_{r \to \infty} \frac{N_{r}R}{(rh)^{4}} \\
& = & R \cdot \limsup_{r \to \infty}\frac{(r+2)(r+1)^{3}}{(rh)^{4}} \left[e^{h/2} + \frac{1}{r+1}\right]\left[e^{h}+\frac{1}{r+1}\right]
\left[e^{3h/4}+\frac{1}{r+1}\right]\\
& = & R \cdot \lim_{r \to \infty}\frac{(r+2)(r+1)^{3}}{(rh)^{4}} \left[e^{h/2} + \frac{1}{r+1}\right]\left[e^{h}+\frac{1}{r+1}\right]
\left[e^{3h/4}+\frac{1}{r+1}\right]\\
& = & R \cdot \frac{e^{9h/4}}{h^{4}} < \infty.
\end{eqnarray*}

A similar argument shows the equivalence of (i') and (ii').
\end{proof}


\section{A Suitable Class of Shearlet Generators}
\label{sec:B0}

Intuitively, the existence of a lower frame bound for a weighted shearlet system
$\cSH(\psi,\Lambda,w)$ should be closely related to a positive lower weighted
shearlet density, as well as the existence of an upper frame bound should be
closely related to a finite density, both for any generator $\psi \in L^2(\RR^2)$.
It will turn out that this is indeed true
for the upper frame bound. However, for a paralleling result for the lower
frame bound, we have to restrict the class of generators to a class such that
the associated weighted shearlet system satisfies a certain approximation
property (see Section \ref{sec:approximation}). Certainly, we require that
this class contains a `sufficiently' large class of generators, wherefore, in
this section, we introduce this class and study decay properties of the
associated continuous shearlet transforms.

\subsection{Amalgam Spaces on the Shearlet Group}

To impose decay conditions of continuous shearlet transforms, we first introduce
so-called amalgam spaces for the group $\SS$, which amalgamate a particular local
and a particular global behavior, each measured by membership in some $L^p$-space.
A comprehensive general theory of amalgam spaces on locally compact groups was
introduced in \cite{Fei80}. For a more expository introduction to Wiener amalgam
spaces on the real line with extensive references, we refer to \cite{Hei03}.

To measure local behavior of functions on $\SS$, for $h = 1$, we consider the
set of boxes {${\{Q_1(e^{j}, e^{-1/4}k, e^{-1/2}m)\}_{j, k \in \ZZ, m \in \ZZ}}$}
associated with the relatively separated (see Lemma \ref{lem:covering}) set of
center points {${X = \{(e^{j}, e^{-1/4}k, e^{-1/2}m)\}_{ j,k \in \mathbb{Z}, m
\in \mathbb{Z}^{2}}}$}.

Now, for each $1\le p<\infty$, the amalgam spaces $W_{\mathbb{S}}(L^{\infty}, L^{p})$
on the shearlet group $\mathbb{S}$ can be defined as mixed-norm spaces consisting of
functions $f: \mathbb{S} \to \mathbb{C}$ by the following

\begin{definition}
For $1\le p<\infty$, the \textit{amalgam space on the shearlet group} $\mathbb{S}$ is
defined by
\[
W_{\mathbb{S}}(L^{\infty}, L^{p}) = \left\{ f\in L^{p}(\mathbb{S}): \|f\|_{W_{\mathbb{S}}(L^{\infty}, L^{p})}
< \infty \right\},
\]
where the norm is given by
\[
\|f\|_{W_{\mathbb{S}}(L^{\infty}, L^{p})}
:= \left(\sum_{j,k\in\Z}\sum_{m \in \Z^{2}} \|f \cdot \chi_{Q_1(e^{j}, e^{-1/4}k, e^{-1/2}m)}\|_{\infty}^{p}\right)^{1/p}.
\]
\end{definition}


\subsection{The Space $\mathcal{B}_0$.}
We now introduce a class of shearlet generators, whose associated continuous shearlet transforms
belong to the amalgam space $W_{\mathbb{S}}(C, L^{1})$ ($= W_{\mathbb{S}}(L^{\infty}, L^{1}) \cap C(\mathbb{S})$, where
$C(\mathbb{S})$ is a set of continuous functions on $\mathbb{S}$). Intuitively, the generators should
have sufficient decay in both space and frequency, which is formalized in the next definition.
Notice that here we draw from ideas presented in \cite{HK07}.

\begin{definition} \label{defi_B0}
Let $\mathcal{B}_{0}$ denote the space of Schwartz functions which
satisfy
\begin{itemize}
 \item[(i)] $|\psi(x)| \le \frac{C}{(1+\|x\|_{\infty}^{2})^{\alpha}}$, \hspace{0.5cm}
 where $C>0$ and $\alpha > \frac{3}{2}$,
 \item[(ii)] $\supp\, \hat{\psi} \wq{\subset} \{[-a_{1}, -a_{0}] \cup [a_{0}, a_{1}]\}\times
 [-b,b]$, $0< a_{0}< a_{1}$ and $b>0$ with
 \[ |\hat{\psi}(\xi)| \le \frac{\xi_{1}^{2\beta}}{(1+\|\xi\|_{\infty}^{2})^{2\beta}}, \qquad
 \textrm{where}\; \;\; \beta > 4\alpha +\gk{2}.\]
\end{itemize}
\end{definition}

The following two lemmata indicate the effect the decay conditions (i) and (ii) imposed on a shearlet
generator have on the decay of the associated continuous shearlet transform. These will
be essential ingredients for proving membership of the associated continuous shearlet transform
to $W_{\mathbb{S}}(C, L^{1})$.

\begin{lem}\label{lem:decay1}
Let $f, \psi \in L^2(\RR^2)$ satisfy Definition \ref{defi_B0}(i). Then
\[
|\mathcal{SH}_{\psi}f(a,s,t)| \le Ca^{3/4}
\frac{\max\{1,d^{2}\}}{\left[1+\left\| \frac{A_{a}^{-1}S_{s}^{-1}t}{\max\{1,d\}}\right\|_{\infty}^{2}\right]^{\alpha - 1/2}},
\quad \mbox{for all }(a,s,t) \in \mathbb{S},
\]
where $d^{2} = \gk{(2+\frac{s^2}{a}) \cdot} \max \left\{ \frac{1}{a^{2}}, \frac{1}{a}\right\}$.
\end{lem}

\begin{proof}
This result follows from \cite[Lem. 4.5]{DKST09} and the relation between $\mathcal{SH}_{\psi}$ and
$\widetilde{\mathcal{SH}}_{\psi}$ given by Lemma \ref{lem:relation_SH_tildeSH}.
\end{proof}

\begin{lem}\label{lem:decay2}
Let $f, \psi \in L^2(\RR^2)$ satisfy Definition \ref{defi_B0}(ii). Then
\[
\left|\mathcal{SH}_{\psi}f(a,s,t)\right| \le C a^{3/4} \frac{a^{3\beta/2}}{(1+a^{2})^{\beta}(\sqrt{a}+|s|)^{\beta}},
\quad \mbox{for all }(a,s,t) \in \mathbb{S}.
\]
\end{lem}

\begin{proof}
This result follows from \cite[Lem. 4.6]{DKST09} and the relation between $\mathcal{SH}_{\psi}$ and
$\widetilde{\mathcal{SH}}_{\psi}$ given by Lemma \ref{lem:relation_SH_tildeSH}.
\end{proof}

The next theorem now proves what was claimed before. We present the lengthy, very technical proof
in Subsection \ref{subsubsec:b0}.

\begin{theorem}\label{theo:b0}
If $f, \psi \in \mathcal{B}_{0}$, then $\mathcal{SH}_{\psi}f \in W_{\mathbb{S}}(C,L^{1})$.
\end{theorem}


\section{Approximation Properties of Weighted Shearlet Frames}
\label{sec:approximation}

The aim of this section is two-fold: (1) approximation properties of shearlet
frames shall be formalized and studied, and (2) a particular approximation property
shall be introduced to serve as a hypothesis under which we can establish a connection
between the existence of a lower frame bound for a shearlet system $\cSH(\psi,\Lambda)$
and a positive lower shearlet density.

We should mention that the to be introduced so-called Homogeneous Approximation Property (HAP)
was already extensively studied in the case of Gabor systems (for a survey, see \cite{Hei07}),
in the wavelet case (see \cite{HK07}), and also for general coherent frame (see \cite{Gro08}).
However, also the results for general coherent frames do not cover the results we
require, since our aim is to establish a link to conditions on shearlet density,
whereas the density aspect is not the focus of \cite{Gro08}. Although many proofs still
follow the basic line of argumentation in those papers -- however much more technical due
to the 4D-parameter space of shearlet -- we include all proofs for the convenience of the
reader sufficiently detailed. To prove the HAP for shearlet frames provided that the
generator belongs to $\cB_0$ might intuitively not be that surprising, however, to actually prove the
results requires hard technical work as can be seen below.

\subsection{The Strong and Weak Homogeneous Approximation Property}

The approximation property we will investigate analyzes the ability of a shearlet
frame to homogeneously approximate functions. More precisely, if a function can be approximated
by a finite collection of elements from a shearlet frame defined by a selection of
a finite set of parameters, then the via the action of $\SS$ translated function can
be approximated with the same accuracy by the finite collection of shearlet elements
associated with the similarly translated finite set of parameters. The following
definition makes these rough ideas precise.

\begin{definition} \label{defi:HAP}
Let $\psi \in L^{2}(\R^{2})$ and $\Lambda \subset \mathbb{S}$ be such that
$\mathcal{SH}(\psi,\Lambda)$ is a shearlet frame for $L^{2}(\R^{2})$, and
let $\{ \widetilde{\sigma(a,s,t)\psi}:(a,s,t) \in \Lambda\}$ denote its dual
frame. For each $h>0$ and $(p,q,r) \in \mathbb{S}$, define
\[
W(h,(p,q,r)) = \spann\{\widetilde{\sigma(a,s,t)\psi}:(a,s,t) \in Q_{h}(p,q,r) \cap \Lambda\}.
\]
\begin{itemize}
\item[(a)] The frame $\mathcal{SH}(\psi,\Lambda)$ is said to possess the \emph{Weak HAP},
if for each $f \in L^{2}(\R^{2})$ and for all $\epsilon >0$ there exists some $R = R(f,\epsilon) > 0$
such that for all $(p,q,r) \in \mathbb{S}$,
\[
\mbox{\rm dist}\left(\sigma(p,q,r)f, W(R,(p,q,r))\right) < \epsilon.
\]
\item[(b)] The frame $\mathcal{SH}(\psi,\Lambda)$ is said to possess the \emph{Strong HAP},
if for each $f \in L^{2}(\R^{2})$ and for all $\epsilon >0$ there exists some $R = R(f,\epsilon) > 0$
such that for all $(p,q,r) \in \mathbb{S}$,
\[
\Big\| \sigma(p,q,r)f \;\; - \sum_{(a,s,t)\in Q_{R}(p,q,r)\cap\Lambda} \langle \sigma(p,q,r)f,
\sigma(a,s,t)\psi \rangle \widetilde{\sigma(a,s,t)\psi}\Big\|_{2} < \epsilon.
\]
\end{itemize}
\end{definition}

The main result showing that the selected class of shearlet generators $\cB_0$ forces
the associated shearlet system to even satisfy the Strong HAP is our next goal. For
this, we require the following technical lemma, whose proof is presented in
Subsection \ref{subsubsec:lemmahap}.

\begin{lem}\label{lem:hap}
Let $\psi, g \in \mathcal{B}_{0}$ and $\Lambda \subset \mathbb{S}$ such that $ \mathcal{SH}(\psi,\Lambda)$
is a frame for $L^{2}(\R^{2})$. For any $\epsilon, \delta > 0$, there exists some $R=R(g,\epsilon) > 0$ such
that, for each $(p,q,r) \in \mathbb{S}$,
\[
\sum_{(x,y,z) \in (p,q,r)^{-1}\Lambda \setminus Q_{R}} \left| \mathcal{SH}_{\psi}g(x,y,z)\right|^{2} \le \epsilon.
\]
\end{lem}

Using this lemma, we obtain the following

\begin{theorem}
\label{theo:hap}
Let $\psi \in \mathcal{B}_{0}$ and $\Lambda \subset \mathbb{S}$ be such that $\mathcal{SH}(\psi,\Lambda)$ is a
frame for $L^{2}(\R^{2})$. Then the shearlet frame $\mathcal{SH}(\psi,\Lambda)$ satisfies the Strong HAP.
\end{theorem}

\begin{proof}
First we will prove  that the shearlet frame $\mathcal{SH}(\psi,\Lambda)$ possesses the Strong HAP for any
function $g \in \mathcal{B}_{0}$. For this, let $g \in \mathcal{B}_{0}$ and $\epsilon > 0$. Let $A$ denote
the lower frame bound of $\mathcal{SH}(\psi,\Lambda)$. It is well-known that then the upper frame bound of
the dual frame $\{\widetilde{\sigma(a,s,t)\psi}\}_{(a,s,t) \in \SS}$ is $\frac{1}{A}$. By considering the
frame expansion
\[
\sigma(p,q,r)g = \sum_{(a,s,t) \in \Lambda} \langle \sigma (p,q,r)g, \sigma(a,s,t)\psi \rangle \widetilde{\sigma(a,s,t)\psi},
\]
we obtain
{\allowdisplaybreaks
\begin{eqnarray} \nonumber
\lefteqn{\Big\|\sigma(p,q,r)g - \sum_{(a,s,t) \in Q_{R}(p,q,r)\cap\Lambda} \langle \sigma(p,q,r)g, \sigma(a,s,t)\psi\rangle
\widetilde{\sigma(a,s,t)} \psi \Big\|_{2}^{2}}\\ \nonumber
& = & \Big\|\sum_{(a,s,t) \in \Lambda\setminus Q_{R}(p,q,r)} \langle \sigma(p,q,r)g, \sigma(a,s,t)\psi\rangle
\widetilde{\sigma(a,s,t)\psi}\Big\|_{2}^{2}\\ \nonumber
& \le & \frac{1}{A} \sum_{(a,s,t) \in\Lambda\setminus Q_{R}(p,q,r)}\left|
\langle g, \sigma\left((p,q,r)^{-1}\cdot(a,s,t)\right)\psi\rangle \right|^{2}\\ \nonumber
& = & \frac{1}{A} \sum_{(a,s,t) \in\Lambda\setminus Q_{R}(p,q,r)}\left|\mathcal{SH}_{\psi}g\left((p,q,r)^{-1}\cdot(a,s,t)\right)
\right|^{2}\\ \nonumber
& = & \frac{1}{A} \sum_{(x,y,z) \in (p,q,r)^{-1}\Lambda\setminus Q_{R}}\left|\mathcal{SH}_{\psi}g\left(x,y,z\right) \right|^{2}\\ \label{eq:lastclaimhap}
& \le & \epsilon.
\end{eqnarray}
}
The last inequality follows from Lemma \ref{lem:hap}.

Now suppose that $f$ is any function in $L^{2}(\R^{2})$. Since $\mathcal{B}_{0}$ is dense in
$L^{2}(\R^{2})$, for each $\delta > 0$, there exists some $g \in \mathcal{B}_{0}$ such that
$\|f-g\|_{2} < \delta$.
Using standard arguments the general claim can now be established from here by using \eqref{eq:lastclaimhap}.
\end{proof}


\subsection{The Comparison Theorem}
\label{subsec:comparison}

Suppose $\mathcal{SH}(\psi,\Lambda)$ is a given shearlet frame whose density we want to study.
We assume we know that $\mathcal{SH}(\psi,\Lambda)$ satisfies the Weak HAP. The main idea
is now to compare this density with the density of a \emph{reference shearlet frame}
$\mathcal{SH}(\phi,\Delta)$. In fact, we derive the following comparison theorem.

\begin{theorem}\label{theo:comparison}
Let $\psi, \phi \in L^{2}(\R^{2})$ and $\Lambda, \Delta \subset \mathbb{S}$ be such that
$\mathcal{SH}(\psi,\Lambda)$ and $\mathcal{SH}(\phi,\Delta)$ are frames for $L^2(\RR^2)$
and $\mathcal{SH}(\psi,\Lambda)$ satisfies the Weak HAP. Let $B$ be a upper frame bound of $\mathcal{SH}(\phi,\Delta)$
and $C = \|\phi\|_{2}$. Then, for each $\epsilon > 0$, there exists a positive constant R such that
\[
\frac{C(C-\epsilon)}{B(e^{\frac{R}{2}}+Re^{\frac{R}{2}})^{4}}\cD^{-}(\Delta)
\le \cD^{-}(\Lambda)\quad \mbox{and} \quad
\frac{C(C-\epsilon)}{B(e^{\frac{R}{2}}+Re^{\frac{R}{2}})^{4}}\cD^{+}(\Delta)
\le \cD^{+}(\Lambda).
\]
\end{theorem}

The main idea of the proof is the double-projection method introduced in \cite{RS95},
which will be briefly described in the sequel. The detailed proof of Theorem \ref{theo:comparison} is presented in
Subsection \ref{subsubsec:comparison}.

Denoting the dual frame for $\mathcal{SH}_{\psi}(\Lambda)$ by $\{\widetilde{\sigma(a,s,t)\psi}: (a,s,t) \in \Lambda\}$,
for each $h>0$ and $(p,q,r) \in \mathbb{S}$, we consider the finite-dimensional subspaces:
\begin{eqnarray*}
W(h,(p,q,r))
& = & \spann\{\widetilde{\sigma(a,s,t)\psi} : (a,s,t) \in Q_{h}(p,q,r) \cap \Lambda\}\\
V(h,(p,q,r))
& = & \spann\{\sigma(a,s,t)\phi : (a,s,t) \in Q_{h}(p,q,r)\cap \Delta\}.
\end{eqnarray*}
For any fixed $\epsilon > 0$, let $R=R(\phi, \epsilon)$ be the value from Definition \ref{defi:HAP}(a),
and let $(a,s,t) \in Q_{h}(p,q,r)$. For $(x,y,z) \in Q_{R}(a,s,t) \cap \Lambda$, we have
\[
(x,y,z) \in Q_{R}(a,s,t) \cap \Lambda
\subset (p,q,r)Q_{h}Q_{R} \cap \Lambda
\subseteq (p,q,r)Q_{R+he^{\frac{R}{2}}+Rhe^{\frac{R}{4}}}\cap \Lambda,
\]
which implies
\begin{equation} \label{eq:compare1}
W(R,(a,s,t)) \subset W(R+he^{\frac{R}{2}}+Rhe^{\frac{R}{4}}, (p,q,r)).
\end{equation}

Let now $P_{V}$ and $P_{W}$ be the orthogonal projections of $L^{2}(\R^{2})$ onto $V(h,(p,q,r))$ and
$W(R+he^{\frac{R}{2}}+Rhe^{\frac{R}{4}}, (p,q,r))$, respectively. The main idea is to consider
the positive, self-adjoint operator
\[
T = P_{V}P_{W}P_{V} : L^{2}(\R^{2}) \to V(h,(p,q,r)),
\]
and derive a lower and upper estimate for $\mbox{\rm tr}[T]$ in terms of $\cD^{\pm}(\Lambda)$ and $\cD^{\pm}(\Delta)$.


\section{Necessary Conditions for Existence of Irregular Shearlet Frames}
\label{sec:necessary}

We can now link geometric properties of the sets of parameters of shearlet systems in terms
of shearlet density with frame properties of the associated systems. In fact, the existence of a lower
frame bound for a shearlet system $\cSH(\psi,\Lambda)$ is closely related to a positive
lower shearlet density provided that this shearlet system satisfies the Weak HAP, and the
existence of an upper frame bound is closely related to a finite density. Also, recall that
provided $\psi \in \cB_0$, the associated shearlet system always satisfies the Weak HAP
by Theorem \ref{theo:hap}, hence statement (ii) below is applicable to a large class
of generators.

\begin{theorem} \label{theo:main}
Let $\psi \in L^{2}(\mathbb{R}^{2})$ be a nonzero function, let $\Lambda$ be a
discrete subset of $\mathbb{S}$, and let $w : \SS \to \RR^+$ be a weight function.
\begin{itemize}
\item[(i)] If $\mathcal{SH}(\psi,\Lambda,w)$ possesses an upper frame bound for $L^{2}(\R^{2})$,
then $\cD^{+}(\Lambda,w) < \infty$.
\item[(ii)] If $\mathcal{SH}(\psi,\Lambda)$ is a frame for $L^{2}(\R^{2})$ and satisfies the Weak HAP,
then $D^{-}(\Lambda) >0$.
\end{itemize}
\end{theorem}

\begin{proof}
(i). We will show that provided $\cD^+(\Lambda,w) = \infty$, for each $N > 0$ there exists some
$g \in  L^{2}(\mathbb{R}^{2})$ with $\|g\|_{2} = 1$ such that
\beq \label{eq:maineq1}
\sum_{(a,s,t) \in \Lambda}|\langle g, w(a,s,t)^{\frac{1}{2}}\sigma(a,s,t)\psi\rangle|^{2} > N.
\eeq
Notice that this implies the non-existence of an upper frame bound for the shearlet system
$\mathcal{SH}_{\psi}(\Lambda,w)$.

First, choose any $\eta \in L^{2}(\mathbb{R}^{2})$ with $\|\eta\|_{2} = 1$. Since the
shearlet transform is continuous, there exist $h > 0$ and $(p,q,r) \in \mathbb{S}$ such that
\[
\delta = \inf_{(a,s,t) \in Q_{h}(p,q,r)}|\mathcal{SH}_{\psi}\eta(a,s,t)| > 0.
\]
Next, let $N > 0$ be arbitrary. Since $D^{+}(\Lambda,w) = \infty$, Proposition \ref{prop:equivalent_D+finite_D-positive}
implies that there exists some $(x,y,z) \in \mathbb{S}$ such that
\[
\#_{w} (\Lambda \cap Q_{h}(x,y,z)) \ge N.
\]
Observe that the function $g:= \sigma((x,y,z)\cdot (p,q,r)^{-1})\eta$ satisfies $g \in L^{2}(\mathbb{R}^{2})$
as well as $\|g\|_{2} = \|\eta\|_{2} = 1$. By using
\[
(a,s,t) \in Q_{h}(x,y,z) \Longrightarrow (p,q,r)\cdot(x,y,z)^{-1}\cdot (a,s,t) \in Q_{h}(p,q,r),
\]
we obtain
\begin{eqnarray*}
\lefteqn{\sum_{(a,s,t) \in \Lambda} |\langle g, w(a,s,t)^{\frac{1}{2}}\sigma(a,s,t)\psi \rangle|^{2}}\\
& \ge & \sum_{(a,s,t) \in \Lambda \cap Q_{h}(x,y,z)} |\langle \sigma((x,y,z)\cdot (p,q,r)^{-1})\eta, w(a,s,t)^{\frac{1}{2}}\sigma(a,s,t)\psi \rangle|^{2} \\
& = & \sum_{(a,s,t) \in \Lambda \cap Q_{h}(x,y,z)} w(a,s,t)|\langle \eta, \sigma((p,q,r)\cdot (x,y,z)^{-1}\cdot (a,s,t))\psi \rangle|^{2} \\
& = & \sum_{(a,s,t) \in \Lambda \cap Q_{h}(x,y,z)} w(a,s,t)| \mathcal{SH}_{\psi}\eta(\underbrace{(p,q,r) \cdot (x,y,z)^{-1}\cdot (a,s,t))}_{\in Q_{h}(p,q,r)}|^{2} \\
& \ge & \#_{w}(\Lambda \cap Q_{h}(x,y,z))\cdot\delta^{2}\\[1ex]
& \ge & N\cdot\delta^{2}.
\end{eqnarray*}
Thus, \eqref{eq:maineq1} is proved.

(ii). Let $\phi \in L^{2}(\R^{2})$ and $\Delta = \{(a^{j}, bk, cm):j,k \in \Z, m \in \Z^{2}\}$ with $a>1$ and $b,c
>0$ be chosen such that the regular shearlet system $\mathcal{SH}_{\phi}(\Delta)$ forms a frame for $L^2(\RR^2)$ with
frame bounds $A$ and $B$. It will be proven in Proposition \ref{prop:oversampled} -- the proof of this proposition
does not require the result under consideration -- that the subset $\Delta$ of the
shearlet group $\SS$ possesses a uniform density, in particular, $\cD^-(\Delta) = \cD^+(\Delta) = \frac{1}{bc^{2}\ln a}$.
Using the constant $C:= \|\phi\|_{2}$ and applying Theorem \ref{theo:comparison} to $\mathcal{SH}_{\psi}(\Lambda)$ and
$\mathcal{SH}_{\phi}(\Delta)$,
\[
\cD^{-}(\Lambda) \ge \frac{1}{bc^{2}\ln a}\cdot\frac{C(C-\epsilon)} {B(e^{\frac{R}{2}}+Re^{\frac{R}{4}})^{4}} > 0,
\]
what was claimed.
\end{proof}


\section{Density Analysis of Various Classes of Shearlet Systems}
\label{sec:examples}

To illustrate applicability of the introduced density analysis and, in particular, of Theorem
\ref{theo:main}, we will study different classes of weighted shearlet systems --
meaning, different classes of sets of parameters -- concerning their potential
to lead to a frame with suitable generator. The classes to be analyzed consist of
oversampled shearlet systems and co-shearlet systems. We also add a discussion
of the (non-)usefulness of density notions based on different (isomorphic) shearlet
groups such as $\tilde{\SS}$.

\subsection{Oversampled Shearlet Systems}

The technique of oversampling was already used for different types of systems; for an
overview with various references, see \cite{HLWW04}. The main objective is to
generate a denser location grid, aiming at a positive lower frame bound (see also
\cite{KKL10}). Customarily, the translation parameter of a regular system -- such as
the regular shearlet systems \eqref{eq:classicalshearlets} -- is oversampled.
We follow this strategy with the following

\begin{definition}
Let $\psi \in L^{2}(\mathbb{R}^{2})$, $a>1$, and $b,c > 0$, and let $\{R_{j,k}\}_{j,k \in
\mathbb{Z}} \subset GL_{2}(\mathbb{R})$. Then we define the {\em oversampled shearlet systems}
generated by $\psi$ with respect to the sequence of matrices $\{R_{j,k}\}_{j,k \in
\mathbb{Z}} \subset GL_{2}(\mathbb{R})$ to be $\cSH(\psi,\Lambda,w)$,
where
\beq \label{eq:defiLambda}
\Lambda = \{(a^{j},bk,cR^{-1}_{jk}m) : j, k \in \ZZ, m \in \ZZ^2\}
\eeq
and
\beq \label{eq:defiw}
w(a^{j},bk,cR^{-1}_{jk}m) = |\det R_{j,k}|^{-1}.
\eeq
\end{definition}

Obviously, this definition includes the regular shearlet systems \eqref{eq:classicalshearlets}
as a special case by choosing $R_{j,k} = Id$ for all $j,k$.

We now study two different types of oversampling matrices $\{R_{j,k}\}_{j,k \in
\mathbb{Z}} \subset GL_{2}(\mathbb{R})$:
\bitem
\item Diagonal matrices: This study includes regular shearlet systems, and, in general,
provides a lattice-oriented oversampling.
\item Shear matrices: These allows the oversampling to be biased towards
a prespecified direction.
\eitem
Interestingly, the upper and lower weighted shearlet density coincide, hence
all oversampled shearlet systems, and, in particular, regular shearlet systems
are associated with a set of parameters with a uniform weighted shearlet density.
Intuitively, this seems plausible by regarding the positioning of the parameters
displayed in Figure \ref{fig:regularshearlets}.

\begin{figure}[ht]
\begin{center}
\includegraphics[height=1.4in]{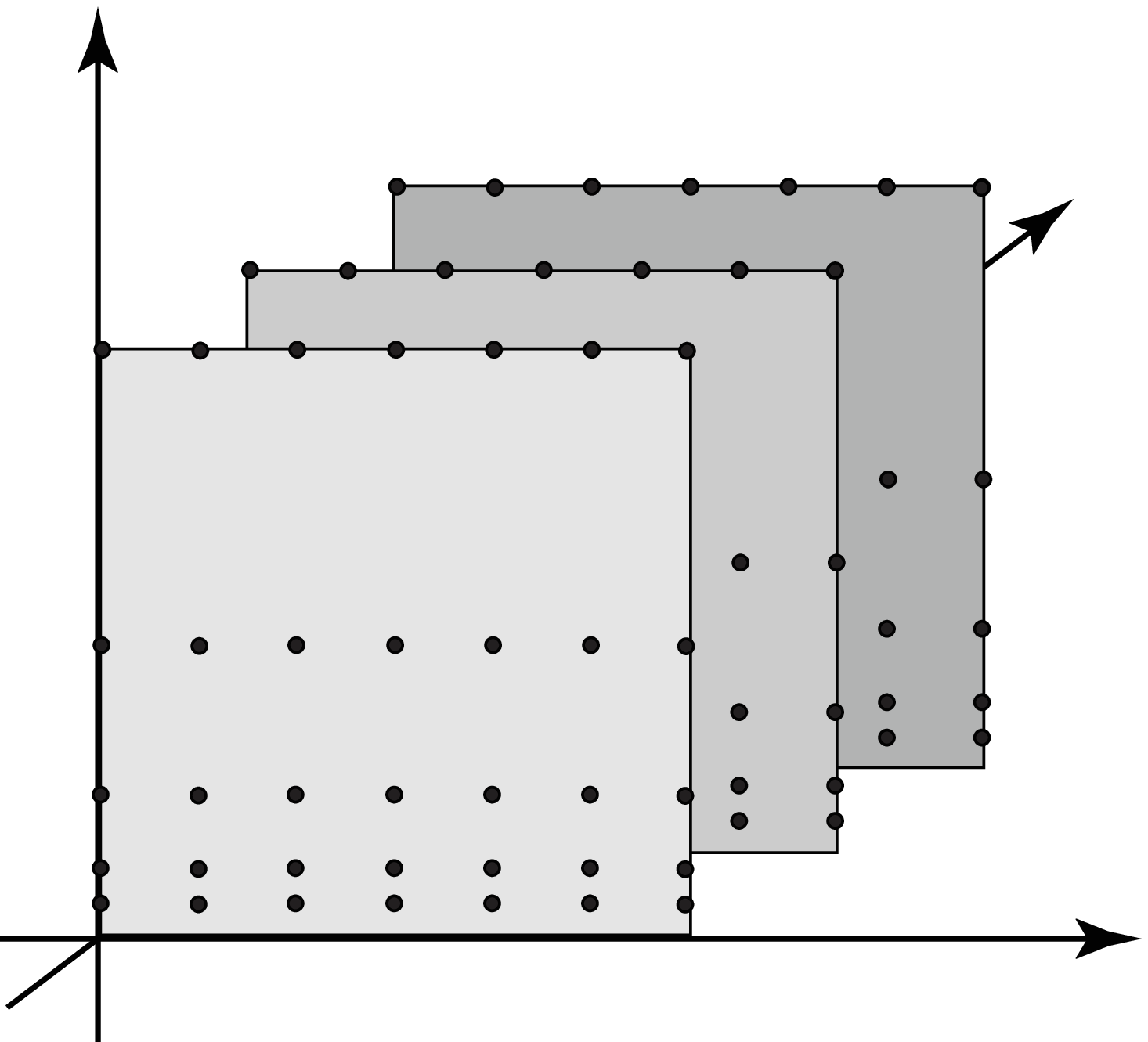}
\put(-8,72){\footnotesize{$s$}}
\put(-110,90){\footnotesize{$a$}}
\put(-10,0){\footnotesize{$t_2$}}
\hspace*{1cm}
\includegraphics[height=1.4in]{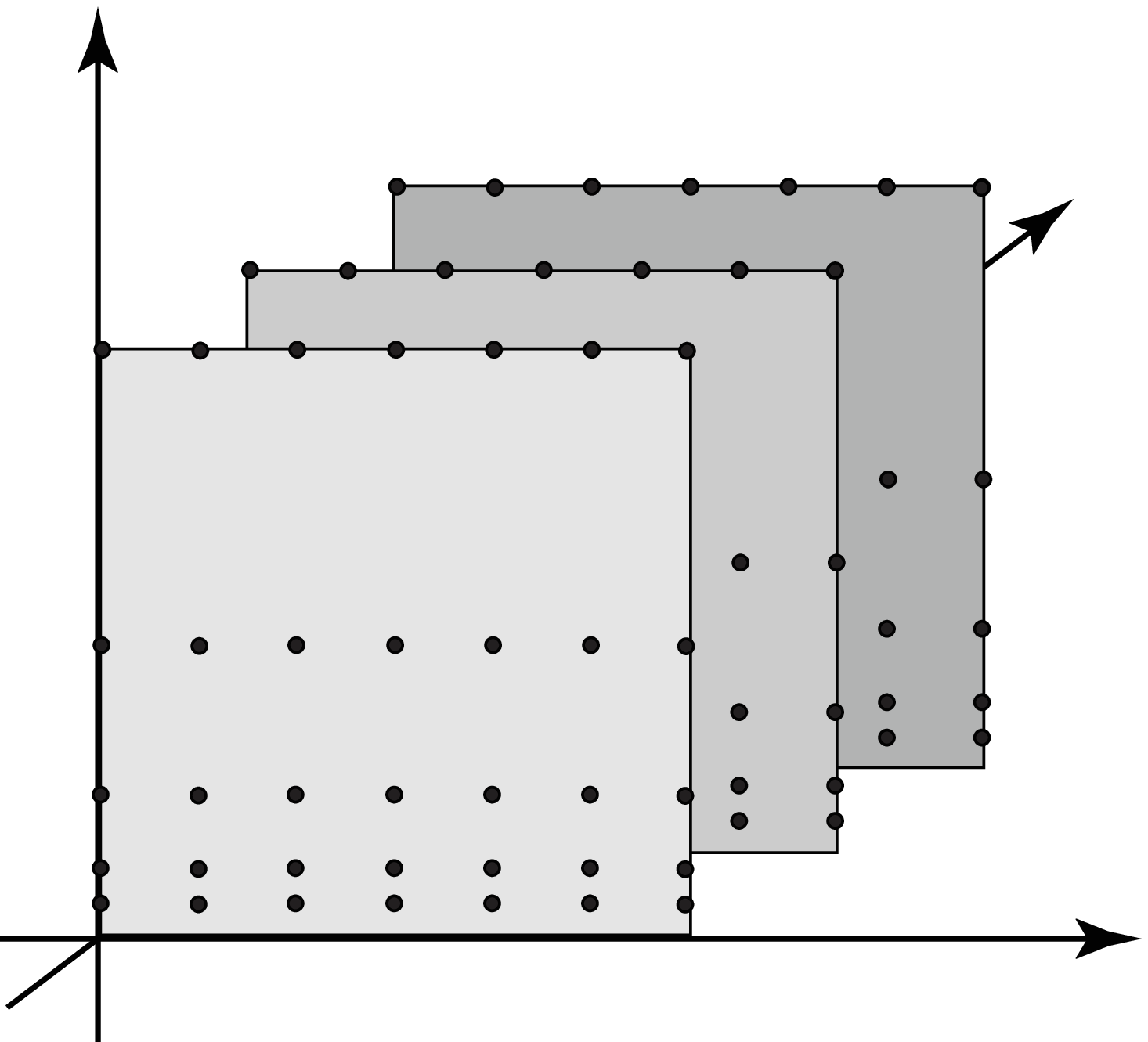}
\put(-8,72){\footnotesize{$t_1$}}
\put(-110,90){\footnotesize{$a$}}
\put(-10,0){\footnotesize{$t_2$}}
\hspace*{1cm}
\includegraphics[height=1.4in]{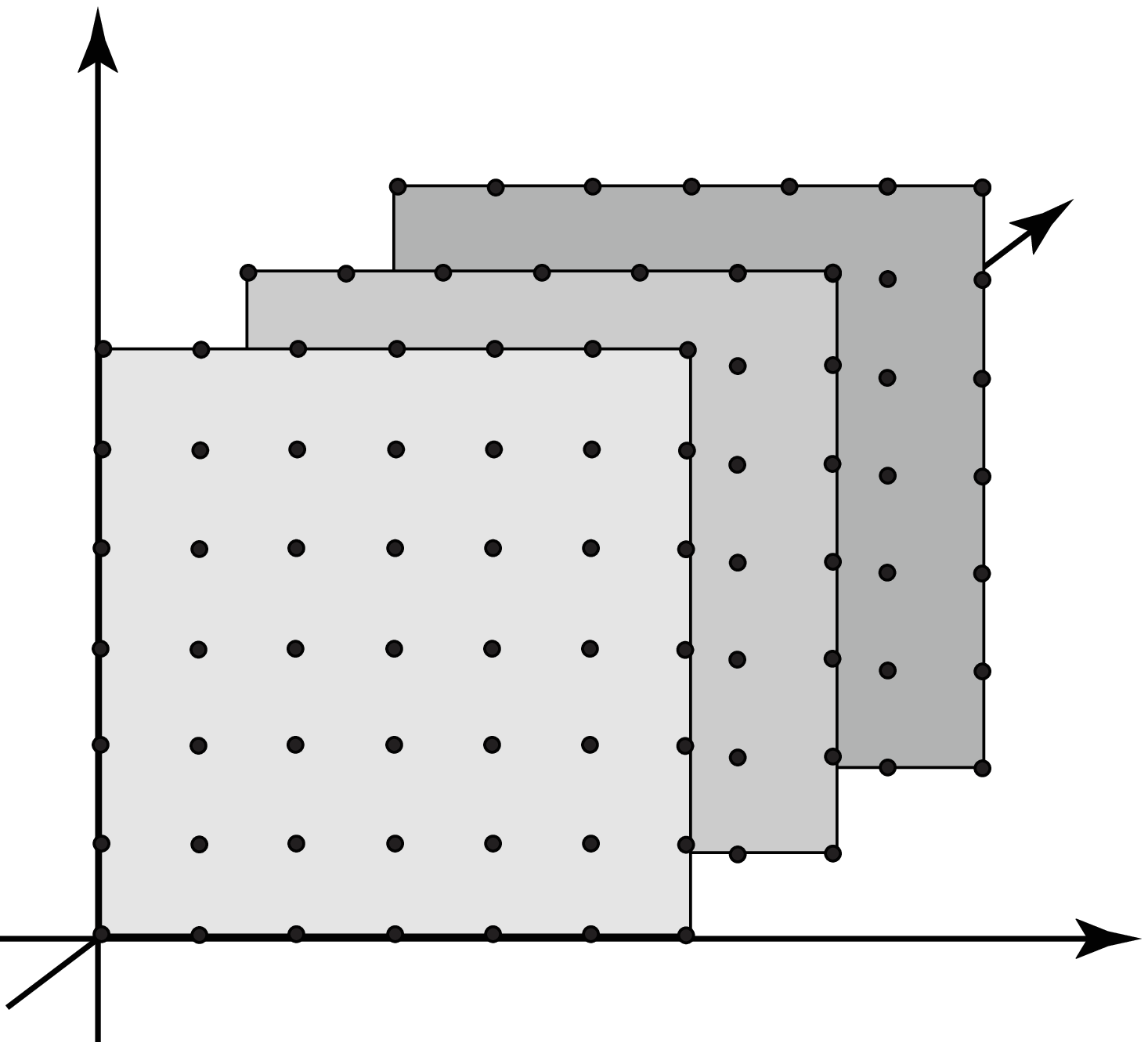}
\put(-8,72){\footnotesize{$s$}}
\put(-110,90){\footnotesize{$t_1$}}
\put(-10,0){\footnotesize{$t_2$}}
\end{center}
\caption{Positioning of the parameters of a regular shearlet system.}
\label{fig:regularshearlets}
\end{figure}

\begin{prop}\label{prop:oversampled}
Let $\psi \in L^{2}(\mathbb{R}^{2})$, $a>1$, and $b,c > 0$, and let $\{R_{j,k}\}_{j,k \in
\mathbb{Z}} \subset GL_{2}(\mathbb{R})$ be one of the following two cases:
\bitem
\item[(a)] $R_{j,k} = $ {\rm diag}$(r_{j,k}^{(1)},r_{j,k}^{(2)})$ with $r_{j,k}^{(1)},r_{j,k}^{(2)} \in \RR$ for all $j,k \in \Z$.
\item[(b)] $R_{j,k} = S_{bk}$ for all $j,k \in \Z$.
\eitem
Let $\cSH(\psi,\Lambda,w)$ be the associated oversampled shearlet system with $\Lambda$
and $w$ being defined by \eqref{eq:defiLambda} and \eqref{eq:defiw}, respectively. Then
$\Lambda$ with weights $w$ has the uniform weighted shearlet density
\[
\cD^{+}(\Lambda,w)=  \cD^{-}(\Lambda,w) = \frac{1}{bc^{2}\ln a}.
\]
\end{prop}

\begin{proof}
We first assume that the sequence of matrices $\{R_{j,k}\}_{j,k \in \mathbb{Z}} \subset GL_{2}(\mathbb{R})$
satisfies (a). Now we fix a center point $(x,y,z) \in \mathbb{S}$ and a size $h > 0$, and
consider the box $Q_h(x,y,z)$. If $(a^{j}, bk, cR_{j,k}^{-1}m) \in \Lambda$ is contained in
$Q_{h}(x,y,z)$, then
\[
(x,y,z)^{-1} \cdot (a^{j}, bk, cR_{j,k}^{-1}m) =
\left( \frac{a^{j}}{x}, bk - \frac{y a^{j/2}}{\sqrt{x}},
cR_{j,k}^{-1}m - S_{bk}A_{a^{j}}A_{x}^{-1}S_{y}^{-1}z \right)\in
Q_{h}.
\]

First, $\frac{a^{j}}{x} \in [e^{-h/2}, e^{h/2})$ implies
\[
\log_{a} x - \frac{h}{2\ln a}  \le j \le \log_{a} x + \frac{h}{2\ln a},
\]
which is satisfied for at least $\frac{h}{\ln a}-1$ and at most $\frac{h}{\ln a}+1$ values of $j$.
Next, $bk - \frac{y a^{j/2}}{\sqrt{x}} \in \left[-\frac{h}{2}, \frac{h}{2} \right)$ yields
\[
\frac{ya^{j/2}}{b\sqrt{x}} - \frac{h}{2b} \le k \le \frac{ya^{j/2}}{b\sqrt{x}} + \frac{h}{2b},
\]
which, for a fixed value of $j$, is fulfilled for at least $\frac{h}{b}-1$ and at most $\frac{h}{b}+1$
values of $k$. Finally,  $c R_{j,k}^{-1} m - \underbrace{S_{bk}A_{a^{j}}A_{x}^{-1}S_{y}^{-1} z}_{=:
(C_{1},C_{2})^{T}} = \gamma \in  \left[-\frac{h}{2}, \frac{h}{2} \right)^{2}$ leads to
\begin{eqnarray} \label{eq:mmm_a}
\gk{\frac{C_{1}}{c}|r_{j,k}^{(1)}|}  -\frac{h}{2c}|r_{j,k}^{(1)}| \le & m_{1} & \le  \gk{\frac{C_{1}}{c}|r_{j,k}^{(1)}|} + \frac{h}{2c}|r_{j,k}^{(1)}|,\\
\gk{\frac{C_{2}}{c}|r_{j,k}^{(2)}|} -\frac{h}{2c}|r_{j,k}^{(2)}| \le & m_{2} & \le  \gk{\frac{C_{2}}{c}|r_{j,k}^{(2)}|} + \frac{h}{2c}|r_{j,k}^{(2)}|. \label{eq:mmm_b}
\end{eqnarray}
For fixed $j$ and $k$, \eqref{eq:mmm_a} is satisfied for at least $\frac{h|r_{j,k}^{(1)}|}{c}-1$
and at most $\frac{h|r_{j,k}^{(1)}|}{c}+1$ values for $m_1$, and \eqref{eq:mmm_b} is satisfied
for at least $\frac{h|r_{j,k}^{(2)}|}{c}-1$ and at most $\frac{h|r_{j,k}^{(2)}|}{c}+1$ values for $m_2$.

Summarizing,
\beq \label{eq:estimateQjxyz}
\#_{w}(\Lambda \cap Q_{h}(x,y,z))= \frac{1}{|\det R_{j,k}|}\cdot\frac{h}{\ln a} \cdot \frac{h}{b}\cdot
\frac{h|r_{j,k}^{(1)}|}{c} \cdot \frac{h|r_{j,k}^{(2)}|}{c}  + \mathcal{O}(h^{3}).
\eeq
Thus,
{\allowdisplaybreaks
\begin{eqnarray*}
\cD^{+}(\Lambda,w)
& = & \limsup_{h \to \infty} \sup_{(x,y,z) \in \mathbb{S}}\frac{\#_{w}(\Lambda \cap
Q_{h}(x,y,z))}{h^{4}}\\
& = & \limsup_{h \to \infty} \left[ \frac{|r_{j,k}^{(1)}\cdot r_{j,k}^{(2)}|}{bc^{2}\ln a|\det R_{j,k}|}
+ \frac{1}{h^{4}}\mathcal{O}(h^{3})\right]\\
& = & \lim_{h \to \infty} \left[ \frac{1}{bc^{2}\ln a} + \frac{1}{h^{4}}\mathcal{O}(h^{3})\right]\\
& = & \frac{1}{bc^{2}\ln a}\\[1ex]
& = & \cD^{-}(\Lambda,w).
\end{eqnarray*}
}
This proves part (a).

Next, we assume that the sequence of matrices $\{R_{j,k}\}_{j,k \in \mathbb{Z}} \subset GL_{2}(\mathbb{R})$
satisfies (b). This proof follows the argumentation of part (a), except that \eqref{eq:mmm_a} and \eqref{eq:mmm_b}
are substituted by
\begin{eqnarray*}
\frac{C_{1}}{c} \gk{+} bkm_{2}-\frac{h}{2c} \le & m_{1} & \le \frac{C_{1}}{c} \gk{+}bkm_{2} + \frac{h}{2c},\\
\frac{C_{2}}{c} -  \frac{h}{2c} \le & m_{2} & \le  \frac{C_{2}}{c}+ \frac{h}{2c},
\end{eqnarray*}
with $C_1$ and $C_2$ defined by $ c \gk{S_{bk}^{-1}} m - \underbrace{S_{bk}A_{a^{j}}A_{x}^{-1}S_{y}^{-1}z}_{=:(C_{1},
C_{2})^{T}}$. This change then leads to
\[
\#(\Lambda \cap Q_{h}(x,y,z)) = \frac{h}{\ln a}\cdot \frac{h}{b}\cdot\left(\frac{h}{c}\right)^{2} + \mathcal{O}(h^{3}).
\]
instead of \eqref{eq:estimateQjxyz}. Thus,
\[
\cD^{+}(\Lambda,w) = \lim_{h \to \infty} \frac{h^{4}}{\left(bc^{2}\ln a\right) h^{4}} = \frac{1}{bc^{2}\ln a} = \cD^{-}(\Lambda,w).
\]
The proposition is proved.
\end{proof}


\subsection{Co-Shearlet Systems}

Co-affine systems were introduced in \cite{GLWW03} for general affine systems in arbitrary
dimension by interchanging dilation and translation in the definition of regular affine
systems. Surprisingly, such systems can never form a frame. Here we use the same
concept to introduce co-shearlet systems by the following

\begin{definition}
Let $\psi \in L^{2}(\mathbb{R}^{2})$, $a>1$, and $b, c > 0$. Then we define {\em co-shearlet
systems} to be $\cSH(\psi,\Lambda,w)$, where
\[
\Lambda = \{(a^{j},bk,S_{bk} A_{a^{j}}cm) : j, k \in \ZZ, m \in \ZZ^2\}
\]
and $w \equiv 1$.
\end{definition}

Notice that these systems indeed arise from interchanging dilation (parabolic scaling and shearing)
and translation in the definition of regular shearlet systems, since changing those operations in
\eqref{eq:classicalshearlets} leads to
\[
\{a^{3j/4} \psi(S_{bk} A_{a^j}(\cdot\,-cm)) : j, k \in \ZZ, m \in \ZZ^2\}.
\]

Also here -- similar to the situation of co-affine systems (see \cite{HK03}) --, the
set of parameters associated with co-shearlet systems leads to extreme density values.
This becomes evident by viewing the positioning of the parameters displayed in Figure
\ref{fig:coshearlets}. The very different distribution compared to Figure
\ref{fig:regularshearlets} will extensively be exploited in the proof of the
following result.

\begin{figure}[ht]
\begin{center}
\includegraphics[height=1.4in]{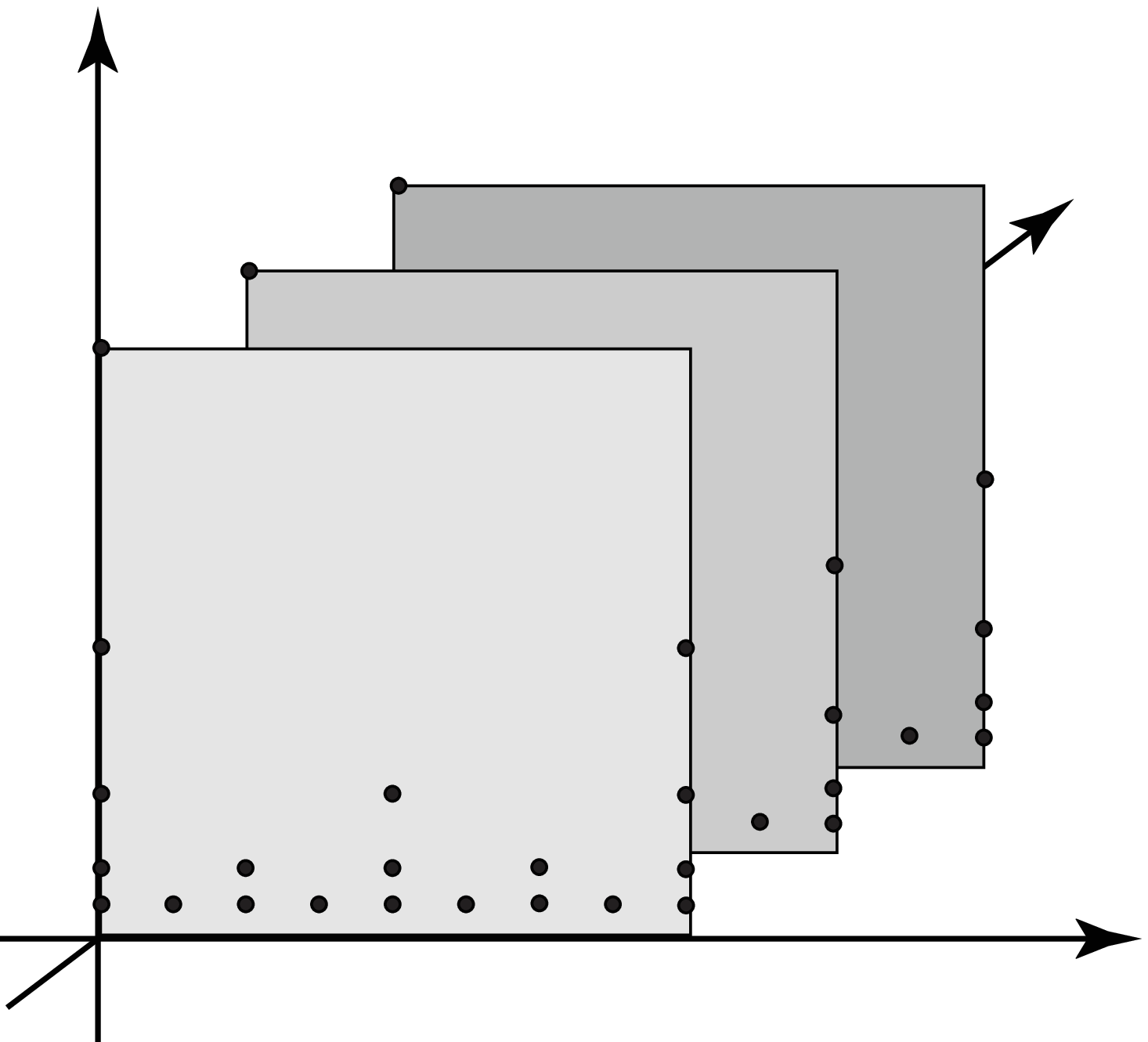}
\put(-8,72){\footnotesize{$s$}}
\put(-110,90){\footnotesize{$a$}}
\put(-10,0){\footnotesize{$t_2$}}
\hspace*{1cm}
\includegraphics[height=1.4in]{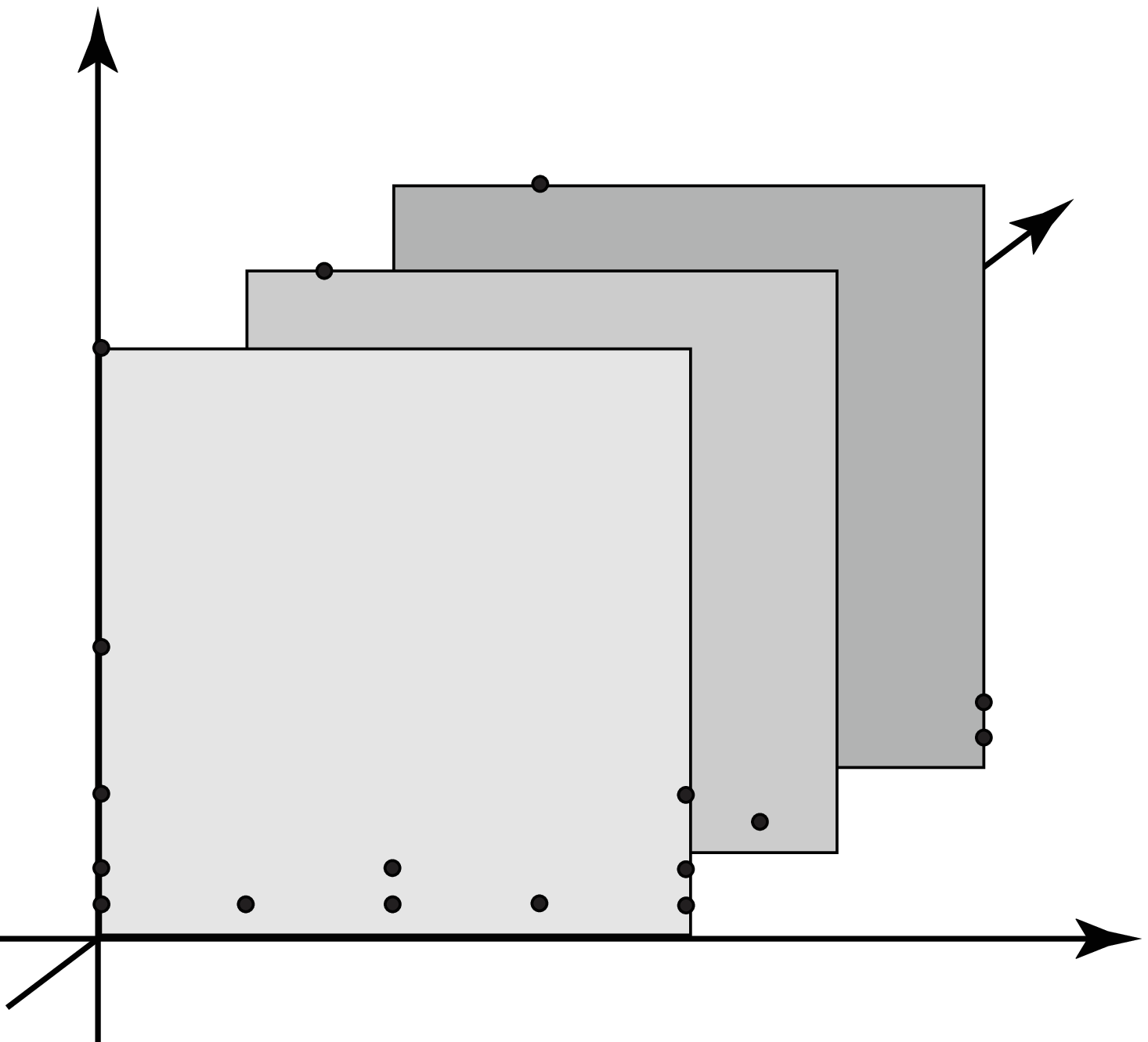}
\put(-8,72){\footnotesize{$t_1$}}
\put(-110,90){\footnotesize{$a$}}
\put(-10,0){\footnotesize{$t_2$}}
\hspace*{1cm}
\includegraphics[height=1.4in]{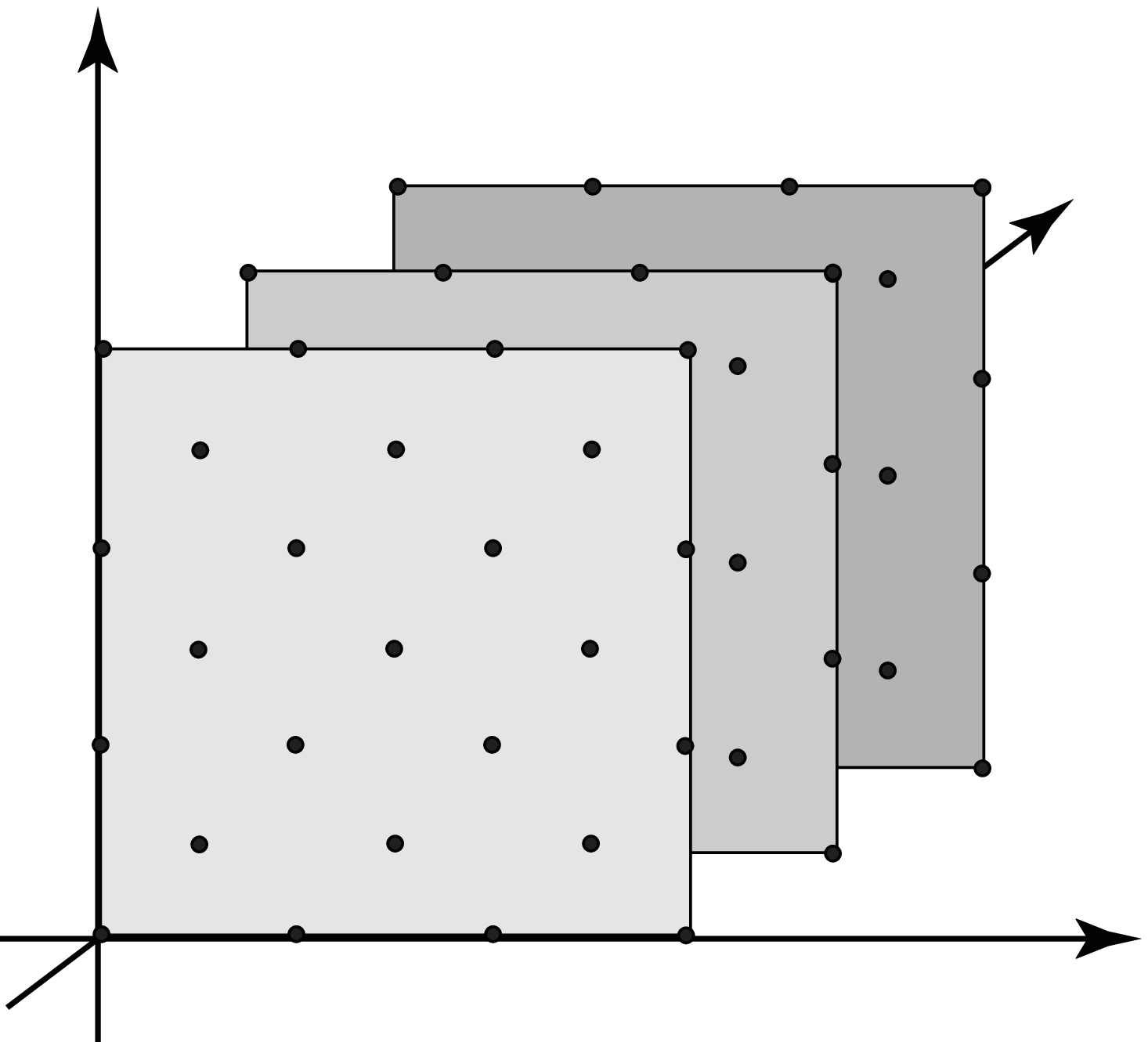}
\put(-8,72){\footnotesize{$s$}}
\put(-110,90){\footnotesize{$t_1$}}
\put(-10,0){\footnotesize{$t_2$}}
\end{center}
\caption{Positioning of the parameters of a co-shearlet system.}
\label{fig:coshearlets}
\end{figure}

\begin{prop}\label{prop:coaff}
Let $\psi \in L^{2}(\mathbb{R}^{2})$, and let $a>1$, and $b,c > 0$.
If $\cSH(\psi,\Lambda,w)$ is the associated co-shearlet system, then
\[
{\cD^{+}(\Lambda,w) =  \infty \quad \mbox{and} \quad \cD^{-}(\Lambda,w) = 0.}
\]
\end{prop}

\begin{proof}
First, fix a center point $(x,y,z) \in \mathbb{S}$ and a size $h > 0$, and
consider the box $Q_h(x,y,z)$. If $(a^{j}, bk, cR_{j,k}^{-1}m) \in Q_{h}(x,y,z)$, then
\[
(x,y,z)^{-1} \cdot (a^{j}, bk, S_{\gk{bk}} A_{a^{j}}cm) =
\left( \frac{a^{j}}{x}, bk - \frac{y a^{j/2}}{\sqrt{x}},
S_{\gk{bk}} A_{a^{j}}cm - S_{bk}A_{a^{j}}A_{x}^{-1}S_{y}^{-1}z \right)\in Q_{h}.
\]

Similar as in the proof of Proposition \ref{prop:oversampled}, $\gk{\frac{a^{j}}{x}} \in [e^{-h/2}, e^{h/2})$ implies
\[
\log_{a} x - \frac{h}{2\ln a}  \le j \le \log_{a} x + \frac{h}{2\ln a},
\]
which is satisfied for approximately $\frac{h}{\ln a}$ values of $j$. (Note that we don't take
term $\pm 1$ into account anymore (cf. the proof of Proposition \ref{prop:oversampled}), since
these only contribute a lower power of $h$ when considering the limit.)
Also, $bk - \frac{y a^{j/2}}{\sqrt{x}} \in \left[-\frac{h}{2}, \frac{h}{2} \right)$ yields
\[
\frac{ya^{j/2}}{b\sqrt{x}} - \frac{h}{2b} \le k \le \frac{ya^{j/2}}{b\sqrt{x}} + \frac{h}{2b},
\]
which, for a fixed value of $j$, is fulfilled for approximately $\frac{h}{b}$
values of $k$.

However, $S_{bk}A_{a^{j}}cm - S_{bk}A_{a^{j}}A_{x}^{-1}S_{y}^{-1} z = \gamma \in \left[-\frac{h}{2},
\frac{h}{2} \right)^{2}$ implies
\[
m = \frac{1}{c} \Big[ A_{a^{-j}}S_{-bk}\gamma + \underbrace{A_{x}^{-1}S_{y}^{-1} z}_{=:(C_{1},C_{2})^{T}} \Big],
\]
which can be rewritten as
\[
m_{1} = C_{1} + \frac{a^{-j}\gamma_{1}}{c}- \frac{bka^{-j}\gamma_{2}}{c}
\quad \mbox{and} \quad
m_{2} = C_{2} + \frac{a^{-j/2}}{c}\gamma_{2},
\]
where $\gamma_{1}, \gamma_{2} \in \left[-\frac{h}{2}, \frac{h}{2} \right)$. Thus,
\begin{eqnarray} \label{eq:mmm1}
C_{1} - a^{-j}(1+b|k|)\frac{h}{2c} \le & m_{1} & \le C_{1}+ a^{-j}(1+b|k|)\frac{h}{2c},\\
C_{2} - \frac{a^{-j/2}h}{2c} \le & m_{2} & \le C_{2} + \frac{a^{-j/2}h}{2c}. \label{eq:mmm2}
\end{eqnarray}
For fixed $j$ and $k$, \eqref{eq:mmm1} is satisfied by approximately $a^{-j}(1+b|k|)h/c$ values of
$m_{1}$, and, for fixed $j$, \eqref{eq:mmm2} is satisfied by approximately $a^{-j/2}h/c$ values of
$m_{2}$.

Summarizing our findings,
\[
\#(\Lambda \cap Q_{h}(x,y,z))
= \sum_{j = \lceil \log_{a} x - \frac{h}{4\ln a}\rceil}^{\lfloor \log_{a} x + \frac{h}{4\ln a}\rfloor}
\sum_{k = \lceil \frac{ya^{j/2}}{b\sqrt{x}} - \frac{h}{2b}\rceil}^{\lfloor \frac{ya^{j/2}}{b\sqrt{x}} + \frac{h}{2b}\rfloor}
\Bigl(\lfloor a^{-j}(1+b|k|)\frac{h}{c}\rfloor+ 1\Bigr) \Bigl(\lfloor\frac{a^{-j/2}h}{2c} \rfloor+ 1\Bigr).
\]
By changing $x$, we can make this quantity arbitrary small or large. Figure
\ref{fig:coshearlets} displays this behavior: As $x$ moves towards the $t_{2}$-axis,
this quantity becomes arbitrarily large, and as $x$ moves away from the
$t_{2}$-axis, this quantity becomes arbitrarily small.

Hence we can conclude that $\cD^{+}(\Lambda,w)= \infty$ and $\cD^{-}(\Lambda,w) = 0$.
\end{proof}

Application of Theorem \ref{theo:main} now enables us to conclude the non-existence of
frame bounds for a co-shearlet system from this proposition.

\begin{corollary}
Let $\psi \in L^{2}(\mathbb{R}^{2})$, let $a>1$, and $b,c > 0$, and let $\cSH(\psi,\Lambda,w)$ be
the associated co-shearlet system. Then $\cSH(\psi,\Lambda,w)$ does not possess an upper frame
bound. In particular, the system  $\cSH(\psi,\Lambda,w)$ can not be a frame.
\end{corollary}

\begin{proof}
By Proposition \ref{prop:coaff}, $\cD^{+}(\Lambda,w)= \infty$ and $\cD^{-}(\Lambda,w) = 0$.
Since $\cD^{+}(\Lambda,w)= \infty$, Theorem \ref{theo:main} (i) implies that $\cSH(\psi,\Lambda,w)$ does
not possess an upper frame bound.
\end{proof}


\subsection{Discussion of Density Results for $\tilde{\SS}$}
\label{subsec:other_groups}

Let us now briefly discuss the potential for the shearlet group $\tilde{\SS}$ to
lead to `nice' density properties. In the wavelet case, it was shown in \cite{Kut07a} that
in fact the choice of affine group multiplication has a strong impact on properties
of the density induced by it. Two versions of the affine group were studied, and it
was shown that one did not lead to a uniform density for the regular wavelet systems
as the other did.

To start our discussion, we first observe that using $\tilde{\SS}$ instead of $\SS$,
the notion of upper and lower weighted shearlet density, $\tilde{\cD}^\pm$, say, can
be introduced in a similar way as in Definition \ref{def:dens1}, however using a
left-invariant Haar measure of $\tilde{\SS}$ instead of $\mu_\SS$. With this definition,
all results from Section \ref{sec:geometry}  -- with different constants -- can be proven
using quite similar proofs as the ones presented.

By using Lemma \ref{lem:relation_SH_tildeSH}, it can easily be computed that the set
\[
\tilde{\Lambda} = \{(\,a^{-j},-bk{a}^{-\frac{j}{2}},cS_{-bk\sqrt{a}^{-{j}}}A_{a^{-j}}m \,):\, j,k \in \Z, \, m \in \Z^2\},
\]
 $a>1$,
$b, c > 0$, is a set of parameters such that $\widetilde{\cSH}(\psi,\tilde{\Lambda})$
is a regular shearlet system. Using similar arguments as in the
proof of Propositions \ref{prop:oversampled} and \ref{prop:coaff}, we can conclude
that in fact regular shearlet systems $\widetilde{\cSH}(\psi,\tilde{\Lambda})$ are
{\em not} associated with a uniform density using $\tilde{\cD}^\pm$.

This justifies our choice of $\SS$.


\section{Proofs}
\label{sec:proofs}

\subsection{Proofs of Results from Section \ref{sec:geometry}}

\subsubsection{Proof of Lemma \ref{lem:covering}}
\label{subsubsec:box}

(i). Fix any $(x, y, z) \in \mathbb{S}$. It suffices to prove the existence of some
$(a, s, t) \in Q_{h}$, and $j, k \in \mathbb{Z}$, $ m \in
\mathbb{Z}^{2}$ such that
\begin{eqnarray*}
(x, y, z) & = & (ae^{jh}, s+he^{-h/4}k\sqrt{a}, t+ h e^{-h/2}S_{s}A_{a}m) \\
& = & (e^{jh}, he^{-h/4}k, he^{-h/2}m) \cdot (a, s, t) \in Q_{h}(e^{jh}, he^{-h/4}k, he^{-h/2}m).
\end{eqnarray*}
The three desired equalities
\[
x = ae^{jh}, \quad y = s+he^{-h/4}k\sqrt{a} \quad \mbox{and} \quad z =  t+he^{-h/2}S_{s}A_{a}m
\]
are equivalent to
\begin{eqnarray}
j & = & \frac{\ln x}{h} - \frac{\ln a}{h}, \label{eq:jj1}\\
k & = & \frac{ye^{h/4}}{h\sqrt{a}} - \frac{se^{h/4}}{h\sqrt{a}}, \label{eq:kk1}\\
m_{1} & = & \frac{(z_{1}-t_{1})e^{h/2}}{ha} - \frac{s(z_{2}-t_{2})e^{h/2}}{ha},\label{eq:mm1} \\
m_{2} & = & \frac{(z_{2}-t_{2})e^{h/2}}{h\sqrt{a}}. \label{eq:mm2}
\end{eqnarray}
We now show how to construct $(a, s, t) \in Q_{h}$, and $j, k \in \mathbb{Z}$, $ m \in \mathbb{Z}^{2}$
satisfying those.

Let $j$ be the unique integer contained in the interval $\left[\frac{\ln x}{h} - \frac{1}{2}, \frac{\ln x}{h} + \frac{1}{2} \right)$,
and set $a = x e^{-jh}$. Then $a \in [e^{-h/2}, e^{h/2})$ and \eqref{eq:jj1}
is satisfied. Next, let $k$ be an integer contained in the interval $\left[\frac{ye^{h/4}}{h\sqrt{a}} - \frac{e^{h/4}}{2\sqrt{a}},
\frac{ye^{h/4}}{h\sqrt{a}} + \frac{e^{h/4}}{2\sqrt{a}} \right)$, and set $s = y -he^{-h/4}k\sqrt{a}$. By choice of
$k$, we have $s \in [-\frac{h}{2}, \frac{h}{2})$ and hence \eqref{eq:kk1} is fulfilled. Further, let $m_2$ be an integer
contained in $\left[\frac{z_{2}e^{h/2}}{h\sqrt{a}} - \frac{e^{h/2}}{2\sqrt{a}}, \frac{z_{2}e^{h/2}}{h\sqrt{a}} +
\frac{e^{h/2}}{2\sqrt{a}}\right)$, and define $t_2$ by $t_2 = z_2-m_2h\sqrt{a}e^{-h/2}$. Again, $t_2 \in [-\frac{h}{2}, \frac{h}{2})$,
and we have established \eqref{eq:mm2}. Finally, we choose the integer $m_1$ to be contained in the interval
$\left[\frac{z_{1}e^{h/2}}{ha}-\frac{sz_{2}e^{h/2}}{ha} +\frac{st_{2}e^{h/2}}{ha}-\frac{e^{h/2}}{2a},\right.$
$\left.\frac{z_{1}e^{h/2}}{ha}-\frac{sz_{2}e^{h/2}}{ha}+\frac{st_{2}e^{h/2}}{ha} +\frac{e^{h/2}}{2a} \right)$,
set $t_1 = z_1- m_1hae^{-h/2}+s(z_{2}-t_{2})$, and conclude that $t_1 \in [-\frac{h}{2}, \frac{h}{2})$.
Thus, also \eqref{eq:mm1} holds.

Concluding, $\{Q_{h}(e^{jh}, he^{-h/4}k, he^{-h/2}m) : j, k \in \mathbb{Z}, m
\in \mathbb{Z}^{2}\}$ is a covering of $\mathbb{S}$, i.e., $X$ is $Q_{h}$-dense in $\mathbb{S}$.

(ii). Fix any $(x, y, z) \in \mathbb{S}$, and let $(u, v, w) \in Q_{rh}(x,y,z) \cap Q_{h}(e^{jh}, he^{-h/4}k, he^{-h/2}m)$.
This implies that there exist $(a, s, t) \in Q_{rh}$ and $(a', s', t') \in Q_{h}$ such that
\begin{eqnarray*}
(u, v, w) & = & (x, y, z) \cdot (a, s,t) \\
& = & (ax, s + y \sqrt{a}, t + S_{s}A_{a}z) \in Q_{rh}(x, y, z)
\end{eqnarray*}
and
\begin{eqnarray*}
(u, v, w) & = & (e^{jh}, he^{-h/4}k, he^{-h/2}m) \cdot (a', s', t') \\
& = & (a'e^{jh}, s'+he^{-h/4}k\sqrt{a'}, t'+ he^{-h/2}S_{s'}A_{a'}m) \in Q_{h}(e^{jh}, he^{-h/4}k, he^{-h/2}m).
\end{eqnarray*}

From $ax = a'e^{jh}$ with $a \in [e^{-rh/2}, e^{rh/2})$ and $a' \in [e^{-h/2}, e^{h/2})$, it follows that
\beq \label{eq:condja1}
xe^{-h(r+1)/2} \le e^{jh} \le xe^{h(r+1)/2},
\eeq
and hence
\[
\frac{\ln x}{h} - \frac{r+1}{2} \le j \le \frac{\ln x}{h} + \frac{r+1}{2}.
\]
This is satisfied for at most $r+2$ values of $j$.

Next, $k = \frac{(s-s')e^{h/4}}{h\sqrt{a'}}+ \frac{y\sqrt{a}e^{h/4}}{h\sqrt{a'}}$ with $s \in [-\frac{rh}{2},
\frac{rh}{2})$ and $s' \in [-\frac{h}{2}, \frac{h}{2})$ implies that
\[
\frac{y\sqrt{a}e^{h/4}}{h\sqrt{a'}} - \left(\frac{r+1}{2}
\right)e^{h/2} \le k \le \frac{y\sqrt{a}e^{h/4}}{h\sqrt{a'}} +
\left(\frac{r+1}{2} \right)e^{h/2}. \label{sumk1}
\]
For fixed values of $a \in [e^{-rh/2}, e^{rh/2})$ and $a' \in [e^{-h/2}, e^{h/2})$, this is satisfied
for at most $(r+1)e^{h/2}+1$ values of $k$.

Finally, we need to study the equation
\[
he^{-h/2}m  = \underbrace{A_{a'}^{-1}S_{s'}^{-1}S_{s}A_{a} z}_{=: (C_{1},C_{2})^{T}} + A_{a'}^{-1}S_{s'}^{-1}(t - t'),
\]
which is
\[
\left(\begin{array}{c} he^{-h/2}m_{1} \\ he^{-h/2}m_{2}\end{array}\right)
= \left( \begin{array}{c} C_{1} + \frac{(t_{1}-t'_{1})}{a'} - \frac{s'(t_{2}-t'_{2})}{a'} \\
C_{2}+\frac{(t_{2}-t'_{2})}{\sqrt{a'}}\end{array}\right),
\]
with $t_{1}, t_{2} \in \left[-\frac{rh}{2}, \frac{rh}{2}\right)$ and $t_{1}', t_{2}' \in \left[-\frac{h}{2}, \frac{h}{2}\right)$.
This yields
\begin{eqnarray} \label{sum_m11}
\frac{C_{1}e^{h/2}}{h} -\frac{s'e^{h/2}(t_{2}-t'_{2})}{a'h}- \frac{(r+1)e^{h}}{2} \le
& m_{1} & \le \frac{C_{1}e^{h/2}}{h} -\frac{s'e^{h/2}(t_{2}-t'_{2})}{a'h}+ \frac{(r+1)e^{h}}{2} \\
\frac{C_{2}e^{h/2}}{h} - \frac{(r+1)e^{3h/4}}{2} \le
& m_{2} & \le \frac{C_{2}e^{h/2}}{h} + \frac{(r+1)e^{3h/4}}{2}. \label{sum_m21}
\end{eqnarray}
For fixed values of $a \in [e^{-rh/2}, e^{rh/2})$, $a' \in [e^{-h/2}, e^{h/2})$, $s \in [-\frac{rh}{2}, \frac{rh}{2})$,
and $s' \in [-\frac{h}{2}, \frac{h}{2})$, inequality \eqref{sum_m11} is satisfied for at most $(r+1)e^{h} + 1$ values
of $m_{1}$, and inequality \eqref{sum_m21} is satisfied for at most $(r+1)e^{3h/4} + 1$ values of $m_{2}$.

Summarizing, the box $Q_{rh}(x, y, z)$ can intersect at most
\[
(r+2)(r+1)^{3}\left[e^{h/2}+\frac{1}{r+1}\right]\left[e^{h}+\frac{1}{r+1}\right] \left[e^{3h/4}+\frac{1}{r+1}\right]
\]
sets of the form $Q_{h}(e^{jh}, he^{-h/4}k, he^{-h/2}m)$.

(iii). We first observe that there exist at least $r$ values of $j$ which satisfy \eqref{eq:condja1}.
Further, for fixed value of $j$, there exist at least $(r+1)e^{h/2}$ values of $k$ which satisfy
\eqref{sumk1}. Finally, for fixed values of $j$ and $k$, inequality \eqref{sum_m11} is satisfied by at
least $(r+1)e^{h}$ values of $m_{1}$, and, for fixed value of $j$, inequality(\ref{sum_m21}) is
satisfied by at least $(r+1)e^{3h/4}$ values of $m_{2}$. Thus $Q_{rh}(x, y, z)$ intersects at least
$ r(r+1)^{3}e^{9h/4}$ sets of the form $Q_{h}(e^{jh}, he^{-h/4}k, he^{-h/2}m)$.

\newpage

\subsection{Proofs of Results from Section \ref{sec:B0}}

\subsubsection{Proof of Theorem \ref{theo:b0}}
\label{subsubsec:b0}


Suppose that $f, \psi \in \mathcal{B}_{0}$.
\wq{We will first start with a general decay estimate for $\mathcal{SH}_{\psi}f$, which will be exploited
frequently throughout the proof}. By Lemmata \ref{lem:decay1} and \ref{lem:decay2}, there exists $C_{1}>0$ such that, for all $(a,s,t) \in \mathbb{S}$,
\begin{equation}\label{eq:sh2}
|\mathcal{SH}_{\psi}f(a,s,t)|^{2}
\le C_{1}a^{3/2}\frac{\max\{1,d^{2}\}}{\left[1+\left\|\frac{A_{a}^{-1}S_{s}^{-1}t}{\max \{1,d\}}\right\|_{\infty}^{2}\right]^{\alpha - 1/2}}
\cdot \frac{a^{3\beta/2}}{(1+a^{2})^{\beta}(\sqrt{a}+|s|)^{\beta}}
\end{equation}
where $\beta > 4\alpha +\gk{2}$ and $\alpha > \frac{3}{2}$ and
$d^{2} = \gk{(2+\frac{s^2}{a}) \cdot} \max \left\{ \frac{1}{a^{2}}, \frac{1}{a}\right\}$.

Consider the sum
\beq \label{eq:sumfinite}
\sum_{j\in \Z}\sum_{k \in \Z } \sum_{m \in \Z^{2}} \left\| \mathcal{SH}_{\psi}f \cdot
\chi_{Q_{1}(e^{j},ke^{-1/4}, e^{-1/2}m)}\right\|_{\infty}.
\eeq
It is sufficient to prove that this sum is finite, which implies $\mathcal{SH}_{\psi}f \in W_{\mathbb{S}}(L^{\infty}, L^{1})$.
Since $\mathcal{SH}_{\psi}f$ is continuous, the proof is then complete.

To prove finiteness of \eqref{eq:sumfinite}, we will carefully split this sum into four subsums, and show
finiteness of each one of those. To find a suitable `splitting point' -- and also to
derive estimates which will become useful later --, for some $j,k \in \Z$ and $m \in \Z^{2}$, we consider an element
$(a,s,t) \in Q_{1}(e^{j},ke^{-1/4}, e^{-1/2}m)$.
Then
\[
(a,s,t) = (e^{j}, ke^{-1/4}, e^{-1/2}m)\cdot (x,y,z) = (xe^{j}, y+ke^{-1/4}\sqrt{x}, z+ e^{-1/2}S_{y}A_{x}m),
\]
for some $(x,y,z) \in Q_{1} = [e^{-1/2},e^{1/2})\times \left[-\frac{1}{2}, \frac{1}{2}\right)\times \left[-\frac{1}{2},
\frac{1}{2}\right)^{2}$. This implies
\begin{itemize}
\item[($C_1$)] $e^{j-1/2} \le a \le e^{j+1/2}$,
\item[($C_2$)] $\frac{|k|}{\sqrt{e}}-\frac{1}{2} \le |s| \le |k|+\frac{1}{2}$,
\item[($C_3$)] $ e^{-1/2} \|S_{y}A_{x}m\|_{\infty} - \frac{1}{2} \le \|t\|_{\infty} \le e^{-1/2} \|S_{y}A_{x}m\|_{\infty} + \frac{1}{2}$.
\end{itemize}
From ($C_3$), we conclude that
\[
\|t\|_{\infty}
\ge e^{-1/2} \|S_{y}A_{x}m\|_{\infty} - \frac{1}{2}
\ge e^{-1/2} \frac{\|m\|_{\infty}}{\|S_{y}^{-1}\|_{\infty}\|A_{x}^{-1}\|_{\infty}} - \frac{1}{2}.
\]
Since $\|S_{y}^{-1}\|_{\infty}\|A_{x}^{-1}\|_{\infty} = (1+|y|)\max \left\{\frac{1}{x},\frac{1}{\sqrt{x}} \right\}$,
we conclude for any $\|m\|_{\infty} > \frac{3e}{2}$,
\beq \label{eq:estimatenormt}
\|t\|_{\infty} \ge \frac{2\|m\|_{\infty}}{3e} - \frac{1}{2} \ge \frac{\|m\|_{\infty}}{3e}.
\eeq

This delivers our `splitting point', and we decompose the sum \eqref{eq:sumfinite} according to
\[
\sum_{j\in \Z}\sum_{k \in \Z } \sum_{m \in \Z^{2}} \left\| \mathcal{SH}_{\psi}f \cdot
\chi_{Q_{1}(e^{j},ke^{-1/4}, e^{-1/2}m)}\right\|_{\infty}  = T_1 + T_2 + T_3 + T_4,
\]
where
{\allowdisplaybreaks
\begin{eqnarray*}
T_1 & = & \sum_{j = -\infty}^{0} \sum_{k \in \Z } \sum_{\{m\in \Z^{2}: \|m\|_{\infty}>\frac{3e}{2}\}}
\left\| \mathcal{SH}_{\psi}f \cdot \chi_{Q_{1}(e^{j},ke^{-1/4}, e^{-1/2}m)}\right\|_{\infty},\\
T_2 & = & \sum_{j = 1}^{\infty} \sum_{k \in \Z} \sum_{\{m\in \Z^{2}: \|m\|_{\infty}>\frac{3e}{2} \}}
\left\| \mathcal{SH}_{\psi}f \cdot \chi_{Q_{1}(e^{j},ke^{-1/4}, e^{-1/2}m)}\right\|_{\infty},\\
T_3 & = & \sum_{j = -\infty}^{0} \sum_{k \in \Z} \sum_{\{ m\in \Z^{2}: \|m\|_{\infty}\le\frac{3e}{2}\}}
\left\| \mathcal{SH}_{\psi}f \cdot \chi_{Q_{1}(e^{j},ke^{-1/4}, e^{-1/2}m)}\right\|_{\infty},\\
T_4 & = & \sum_{j = 1}^{\infty} \sum_{k \in \Z} \sum_{\{ m\in \Z^{2}: \|m\|_{\infty}\le \frac{3e}{2}\}}
\left\| \mathcal{SH}_{\psi}f \cdot \chi_{Q_{1}(e^{j},ke^{-1/4}, e^{-1/2}m)}\right\|_{\infty}.
\end{eqnarray*}
}

We now prove that each of the sums $T_1$--$T_4$ is finite.

$T_1$. Suppose that $(a,s,t) \in Q_{1}(e^{j},ke^{-1/4}, e^{-1/2}m)$ with $j\le 0$, $k \in \Z$, and
$m \in \Z^{2}$, $\|m\|_{\infty} > \frac{3e}{2}$. First assume that $j < 0$, hence $a \le 1$. Then
\[
d^{2} = \gk{\left(2+\frac{s^2}{a}\right)} \max \left\{\frac{1}{a^{2}}, \frac{1}{a}\right\} = \gk{\frac{2a +
s^2}{a^3}} > 1,
\]
and hence, also using \eqref{eq:estimatenormt},
\[
1+ \left\| \frac{A_{a}^{-1}S_{s}^{-1}t}{\max\{1,d\}}\right\|_{\infty}^{2}
\ge 1 + \frac{\left\|t\right\|_{\infty}^{2}}{d^{2} \|A_{a}\|_{\infty}^{2} \|S_{s}\|_{\infty}^{2}}
\ge 1 + \frac{a^{3}\left\|m\right\|_{\infty}^{2}}{9e^{2}\gk{(2a+s^2)}a(1+|s|)^{2}}
\ge 1+\frac{a^{2}\|m\|_{\infty}^{2}}{\gk{18}e^{2}(1+|s|)^{4}}.
\]
This together with \eqref{eq:sh2} and ($C_2$) implies
{\allowdisplaybreaks
\begin{eqnarray*}
|\mathcal{SH}_{\psi}f(a,s,t)|^{2}
& \le & C_{1}a^{3/2} \frac{\gk{2}(\sqrt{a}+|s|)^{2}}{a^{3}\left[1+ \frac{a^{2} \|m\|_{\infty}^{2}}{\gk{18}e^{2}(1+|s|)^{4}}\right]^{\alpha -1/2}}
\cdot \frac{a^{3\beta/2}}{(1+a^{2})^{\beta}(\sqrt{a}+|s|)^{\beta}} \\
& \le & C_{1}a^{3/2} \frac{(1+|s|)^{4\alpha -2}}{a^{2\alpha+2}\left[\frac{\gk{18}e^{2}(1+|s|)^{4}}{a^{2}}+ \|m\|_{\infty}^{2}\right]^{\alpha -1/2}}
\cdot \frac{a^{3\beta/2}}{a^{\frac{\beta}{2}-1}\left(1+\frac{|s|}{\sqrt{a}}\right)^{\beta-2}}\\
& \le & C_{1} \frac{a^{\beta-2\alpha + \frac{1}{2}}}{\|m\|_{\infty}^{2\alpha-1}}\cdot \frac{1}{(1+|s|)^{\beta-4\alpha}}\\
& \le & C_{1}\frac{e^{(\beta-2\alpha + \frac{1}{2})j/2}}{\|m\|_{\infty}^{2\alpha-1}}\cdot \frac{1}{(\sqrt{e}+2|k|)^{\beta-4\alpha}}.
\end{eqnarray*}
}
Hence
\[
|\mathcal{SH}_{\psi}f(a,s,t)|
\le C_{1}\frac{e^{(\beta-2\alpha + \frac{1}{2})j/4}}{\|m\|_{\infty}^{\alpha-\frac{1}{2}}}\cdot \frac{1}{(\sqrt{e}+2|k|)^{(\beta-4\alpha)/2}},
\]
and we obtain
{\allowdisplaybreaks
\begin{eqnarray}\nonumber
\lefteqn{\sum_{j = -\infty}^{-1} \sum_{k \in \Z} \sum_{\{m\in \Z^{2}: \|m\|_{\infty}>\frac{3e}{2}\}}
\left\| \wq{\mathcal{SH}_{\psi}f} \cdot \gk{\chi_{Q_{1}(e^{j},ke^{-1/4}, e^{-1/2}m)}} \right\|_{\infty}}\\ \label{eq:lastnumber1}
& \le & C_{1} \sum_{j = -\infty}^{0} e^{(\beta-2\alpha + \frac{1}{2})j/4} \sum_{k \in \Z}\frac{1}{(\sqrt{e}+2|k|)^{(\beta-4\alpha)/2}}
\sum_{\{m\in \Z^{2}: \|m\|_{\infty} > \frac{3e}{2}\}} \|m\|_{\infty}^{-(\alpha-\frac{1}{2})},
\end{eqnarray}
}
which is finite.

If $j=0$, then $1 < a \le \sqrt{e}$, and it is easy to see that, by \eqref{eq:sh2} and ($C_2$),
\[
|\mathcal{SH}_{\psi}f(a,s,t)|
\le \frac{C_1}{\|m\|_{\infty}^{\alpha-\frac{1}{2}}}\cdot \frac{1}{(\sqrt{e}+2|k|)^{(\beta-4\alpha)/2}},
\]
which implies
\beq \label{eq:lastnumber2}
\sum_{k \in \Z} \sum_{\{m\in \Z^{2}: \|m\|_{\infty}>\frac{3e}{2}\}}
\left\| \wq{\mathcal{SH}_{\psi}f} \cdot {\chi_{Q_{1}(1,ke^{-1/4}, e^{-1/2}m)}} \right\|_{\infty} < \infty.
\eeq
Hence, from \eqref{eq:lastnumber1} and \eqref{eq:lastnumber2}, it follows that $T_1$ is finite.

$T_2$. Suppose that $(a,s,t) \in Q_{1}(e^{j},ke^{-1/4}, e^{-1/2}m)$ with $j>0, k \in \Z$, $m \in \Z^{2}$, and
$m \in \Z^{2}$, $\|m\|_{\infty} > \frac{3e}{2}$. Then $a >1$ and $d^{2} = \gk{\frac{2a + s^2}{a^2}}$.

Now we distinguish two cases:
If $d^{2} > 1$, then
{\allowdisplaybreaks
\begin{eqnarray*}
1+ \left\| \frac{A_{a}^{-1}S_{s}^{-1}t}{\max\{1,d\}}\right\|_{\infty}^{2}
& \ge & 1 + \frac{\left\|t\right\|^{2}_\infty}{d^{2} \|A_{a}\|_{\infty}^{2} \|S_{s}\|_{\infty}^{2}} \ge 1 +
\frac{a^{2}\left\|m\right\|^{2}_\infty}{9e^{2}\gk{(2a+s^2)}a^{2}(1+|s|)^{2}}\\
& \ge & 1+\frac{\|m\|_{\infty}^{2}}{\gk{18}e^{2}(\sqrt{a}+|s|)^{4}}.
\end{eqnarray*}
}
By \eqref{eq:sh2}, we hence obtain
{\allowdisplaybreaks
\begin{eqnarray*}
|\mathcal{SH}_{\psi}f(a,s,t)|^{2}
& \le & C_{1} a^{3/2} \frac{\gk{2}(\sqrt{a}+|s|)^{2}}{a^{2}\left[1+ \frac{\|m\|_{\infty}^{2}}{\gk{18}e^{2}(\sqrt{a}+|s|)^{4}}\right]^{\alpha -1/2}}
\cdot \frac{a^{3\beta/2}}{(1+a^{2})^{\beta}(\sqrt{a}+|s|)^{\beta}}\\
& \le & C_{1} \frac{a^{-(\beta+1)/2}(\sqrt{a}+|s|)^{4\alpha}}{\left[\gk{18}e^{2}(\sqrt{a}+|s|)^{4}+\|m\|_{\infty}^{2}\right]^{\alpha -1/2}}
\cdot \frac{1}{(\sqrt{a}+|s|)^{\beta}}\\
& \le & C_{1}  \frac{a^{-(\beta+1)/2}}{\|m\|_{\infty}^{2\alpha-1}} \cdot \frac{1}{(1+|s|)^{\beta-4\alpha}}\\
& \le & C_{1}\frac{e^{-(\beta+1)j/2}}{\|m\|_{\infty}^{2\alpha-1}}\cdot\frac{1}{(\sqrt{e}+2|k|)^{\beta-4\alpha}}.
\end{eqnarray*}
}
This yields
{\allowdisplaybreaks
\begin{eqnarray*}
T_2
& = & \sum_{j = 1}^{\infty} \sum_{k \in \Z}\sum_{m\in \Z^{2}, \|m\|_{\infty}>\frac{3e}{2}}
\left\| \wq{\mathcal{SH}_{\psi}f} \cdot\gk{\chi_{Q_{1}(e^{j},ke^{-1/4}, e^{-1/2}m)}}\right\|_{\infty}\\
& \le & C_{1} \sum_{j = 1}^{\infty} e^{-(\beta+1)j/4}\sum_{k \in \Z}\frac{1}{(\sqrt{e}+2|k|)^{(\beta-4\alpha)/2}}
\sum_{m\in \Z^{2}} \|m\|_{\infty}^{-(\alpha-\frac{1}{2})},
\end{eqnarray*}
}
hence $T_2$ is finite.

If $d^{2} \le 1$, then
\[
1+ \left\|\frac{A_{a}^{-1}S_{s}^{-1}t}{\max\{1,d\}}\right\|_{\infty}^{2}
\ge 1 + \frac{\left\|t\right\|^{2}_\infty}{a^{2}(1+|s|)^{2}}
\ge 1+\frac{\|m\|_{\infty}^{2}}{9e^{2}a^{2}(\sqrt{a}+|s|)^{2}}.
\]
By \eqref{eq:sh2}, we hence obtain
{\allowdisplaybreaks
\begin{eqnarray*}
|\mathcal{SH}_{\psi}f(a,s,t)|^{2}
& \le & C_{1}a^{3/2}\frac{1}{\left[1+ \frac{\|m\|_{\infty}^{2}}{9e^{2}a^{2}(\sqrt{a}+|s|)^{4}}\right]^{\alpha -1/2}}
\cdot\frac{a^{3\beta/2}}{(1+a^{2})^{\beta}(\sqrt{a}+|s|)^{\beta}}\\
& \le & C_{1}\frac{a^{-(\beta-\gk{2}\alpha-1)/2}(\sqrt{a}+|s|)^{4\alpha-2}}{\left[9e^{2}a^{2}(\sqrt{a}+|s|)^{4}+ \|m\|_{\infty}^{2}\right]^{\alpha -1/2}}
\cdot\frac{1}{(\sqrt{a}+|s|)^{\beta}}\\
& \le & C_{1}  \frac{a^{-(\beta-\gk{2}\alpha-1)/2}}{\|m\|_{\infty}^{2\alpha-1}} \cdot \frac{1}{\left(1+|s|\right)^{\beta-4\alpha+2}}\\
& \le & C_{1}\frac{e^{-(\beta-\gk{2}\alpha -1)j/2}}{\|m\|_{\infty}^{2\alpha-1}}\cdot \frac{1}{(\sqrt{e}+2|k|)^{\beta-4\alpha+2}}.
\end{eqnarray*}
}
This yields
\[
T_2
\le C_{1} \sum_{j = 1}^{\infty} e^{-(\beta-\gk{2}\alpha -1)j/4} \sum_{k \in \Z}
\frac{1}{(\sqrt{e}+2|k|)^{(\beta-4\alpha+2)/2}} \sum_{\{m\in \Z^{2}: \|m\|_{\infty} > \frac{3e}{2}\}} \|m\|_{\infty}^{-(\alpha-\frac{1}{2})},
\]
which is finite.

$T_3$. Suppose that $(a,s,t) \in Q_{1}(e^{j},ke^{-1/4}, e^{-1/2}m)$ with $j\le 0, k \in \Z$,
$m \in \Z^{2}$, and $\|m\|_{\infty} < \frac{3e}{2}$. First assume that $j < 0$, hence $a \le 1$. Then, by Lemma \ref{lem:decay2},
\[
|\mathcal{SH}_{\psi}f(a,s,t)|
\le C a^{3/4}\cdot \frac{a^{3\beta/2}}{a^{\frac{\beta}{2}}(1+a^{2})^{\beta} \left(1+\frac{|s|}{\sqrt{a}}\right)^{\beta}}
\le C \frac{a^{\beta+\frac{3}{4}}}{(1+|s|)^{\beta}}
\le C \frac{e^{(\beta+\frac{3}{4})j/2}}{(\sqrt{e}+2|k|)^{\beta}}.
\]
And hence
\begin{eqnarray} \nonumber
\lefteqn{\sum_{j = -\infty}^{0} \sum_{k \in \Z} \sum_{\{m\in \Z^{2}: \|m\|_{\infty}\le\frac{3e}{2}\}} \left\| \wq{\mathcal{SH}_{\psi}f} \cdot
\gk{\chi_{Q_{1}(e^{j},ke^{-1/4}, e^{-1/2}m)}}\right\|_{\infty}}\\ \label{eq:lastnumber3}
& \le & C \, (3e + 1)^{2} \sum_{j = -\infty}^{0} e^{(\beta+\frac{3}{4})j/2}\sum_{k \in \Z} \frac{1}{(\sqrt{e}+2|k|)^{\beta}},
\end{eqnarray}
which is finite.

If $j=0$, then $1 < a \le \sqrt{e}$, and it is easy to see that, by Lemma \ref{lem:decay2},
\[
|\mathcal{SH}_{\psi}f(a,s,t)| \le \frac{C}{(\sqrt{e}+2|k|)^{\beta}},
\]
hence
\beq \label{eq:lastnumber4}
\sum_{k \in \Z} \sum_{\{m\in \Z^{2}: \|m\|_{\infty} \le \frac{3e}{2}\}}
\left\| \wq{\mathcal{SH}_{\psi}f} \cdot {\chi_{Q_{1}(1,ke^{-1/4}, e^{-1/2}m)}} \right\|_{\infty} < \infty.
\eeq
Thus, by \eqref{eq:lastnumber3} and \eqref{eq:lastnumber4}, $T_3$ is finite.

$T_4$. Suppose that $(a,s,t) \in Q_{1}(e^{j},ke^{-1/4}, e^{-1/2}m)$ with $j > 0\; (\textrm{i.e.},\; a >1),\, k \in \Z$,
$m \in \Z^{2}$, and $\|m\|_{\infty} < \frac{3e}{2}$. By Lemma \ref{lem:decay2},
\[
|\mathcal{SH}_{\psi}f(a,s,t)|
\le C a^{3/4} \cdot \frac{a^{3\beta/2}}{a^{2\beta}(a^{-2}+1)^{\beta} \left(\sqrt{a}+ |s|\right)^{\beta}}
\le C a^{-\frac{\beta}{2}+\frac{3}{4}}\cdot \frac{1}{(1+|s|)^{\beta}}
\le C \frac{e^{-(\frac{\beta}{2}-\frac{3}{4})j}}{(\sqrt{e}+2|k|)^{\beta}}.
\]
This implies
\begin{eqnarray*}
T_4
& = &  \sum_{j = 1}^{\infty} \sum_{k \in \Z} \sum_{\{m\in \Z^{2}: \|m\|_{\infty}\le\frac{3e}{2}\}} \left\| \wq{\mathcal{SH}_{\psi}f} \cdot
\gk{\chi_{Q_{1}(e^{j},ke^{-1/4}, e^{-1/2}m)}}\right\|_{\infty}\\
& \le & C \, (3e+1)^{2} \sum_{j = 1}^{\infty} e^{-(\frac{\beta}{2}-\frac{3}{4})j}\sum_{k \in \Z} \frac{1}{(\sqrt{e}+2|k|)^{\beta}},
\end{eqnarray*}
and hence $T_4$ is finite.

Summarizing, we proved that the sum \eqref{eq:sumfinite} is finite, which finishes the proof.

\subsection{Proofs of Results from Section \ref{sec:approximation}}

\subsubsection{Proof of Lemma \ref{lem:hap}}
\label{subsubsec:lemmahap}

Let $\delta>0$ and $R'>1$ be given, and define
\[
R := \left(1+\frac{\delta}{2}\right)^{2}e^{\delta}R' + \delta \left(1+\frac{\delta}{2}\right)^{2}e^{\delta} + \delta.
\]

Our first claim is that for any $(p,q,r) \in \mathbb{S}$,
\begin{equation} \label{eq:clam}
Q_{\delta}\gk{(p,q,r)}\setminus Q_{R} \ne \emptyset \quad \Longrightarrow \quad Q_{\delta}(p,q,r)\cap Q_{R'} = \emptyset,
\end{equation}
a technical ingredient required in the main body of the proof.

Suppose now that there exists some $(a,s,t) \in Q_{\delta}(p,q,r)\setminus Q_{R}$. To prove \eqref{eq:clam}, we
assume that $(x,y,z) \in Q_{\delta}$, and show that this implies
\beq \label{eq:claim11}
 (p,q,r)\cdot(x,y,z) = (px,y+q\sqrt{x}, z + S_{y}A_{x}r) \notin Q_{R'}.
\eeq
First, our hypotheses imply
\begin{equation} \label{eq:lem51}
(a,s,t) \notin Q_{R},
\end{equation}
\begin{equation} \label{eq:lem52}
(p,q,r)^{-1}\cdot(a,s,t) = \left( \frac{a}{p}, s-q\sqrt{\frac{a}{p}}, t-S_{s}A_{a}S_{-\frac{q}{\sqrt{a}}}A_{\frac{1}{p}}r\right)
\in Q_{\delta},
\end{equation}
\begin{equation} \label{eq:lem53}
(x,y,z) \in Q_{\delta}.
\end{equation}
Next we use \eqref{eq:lem51}--\eqref{eq:lem53} to prove \eqref{eq:claim11}.

For this, first suppose that $a> e^{\frac{R}{2}}$.  Then \eqref{eq:lem51}--\eqref{eq:lem53} imply
\[
px = \frac{p}{a}(ax) \ge e^{-\frac{\delta}{2}}e^{\frac{R}{2}}e^{-\frac{\delta}{2}}
= e^{\frac{R}{2}-\delta} \ge e^{\frac{R'}{2}}.
\]
Similarly, if $a < e^{-\frac{R}{2}}$, then $px < e^{-\frac{R'}{2}}$. In either case,
\eqref{eq:claim11} holds.

Next, suppose that $s \ge \frac{R}{2}$. Then, since
\[
y+q\sqrt{x} = y-\left(s - q\sqrt{\frac{a}{p}}\right)\sqrt{\frac{p}{a}}\sqrt{x} + s \sqrt{\frac{p}{a}}\sqrt{x}.
\]
we obtain
\[
y+q\sqrt{x}
\ge -\frac{\delta}{2}-\frac{\delta}{2}\,e^{\frac{\delta}{4}}e^{\frac{\delta}{4}} + \frac{R}{2}\,e^{-\frac{\delta}{4}}e^{-\frac{\delta}{4}}
= \frac{R}{2}\,e^{-\frac{\delta}{2}}-\frac{\delta}{2}\,e^{\frac{\delta}{2}} -\frac{\delta}{2}
> \frac{R'}{2}.
\]
Similarly, if $s < -\frac{R}{2}$, then $y+q\sqrt{x} < -\frac{R'}{2}$. Again, in either case,
\eqref{eq:claim11} is true.

Finally, if $a \in [e^{-\frac{R}{2}},e^{\frac{R}{2}})$ and $s \in [-\frac{R}{2},\frac{R}{2})$, we
aim to prove $\|z+S_{y}A_{x}r\|_{\infty} > \frac{R'}{2}$. Since \eqref{eq:lem52} implies $\|t -
S_{s}A_{a}S_{-\frac{q}{\sqrt{p}}}A_{\frac{1}{p}}r\|_{\infty} \le \frac{\delta}{2}$ and \eqref{eq:lem51}
implies $\|t\|_{\infty} > \frac{R}{2}$, we obtain
\[
\|S_{s}A_{a}S_{-\frac{q}{\sqrt{p}}}A_{\frac{1}{p}}r\|_{\infty}
\ge \|t\|_{\infty} - \|t - S_{s}A_{a}S_{-\frac{q}{\sqrt{p}}}A_{\frac{1}{p}}r\|_{\infty}
\ge \frac{R-\delta}{2}.
\]
Hence
\[
\|r\|_{\infty}
\ge \frac{(R-\delta)}{2\left\|S_{s-q\sqrt{\frac{a}{p}}}\right\|_{\infty} \|A_{\frac{a}{p}}\|_{\infty}}
= \frac{(R-\delta)}{2\left(1+\left|s-q\sqrt{\frac{a}{p}}\right|\right) \max \left\{\frac{a}{p}, \sqrt{\frac{a}{p}}\right\}}
\ge \frac{(R-\delta)}{2e^{\frac{\delta}{2}}\left(1+\frac{\delta}{2}\right)},
\]
where the last inequality follows from \eqref{eq:lem52}. This implies
\[
\|S_{y}A_{x}r\|_{\infty}
\ge \frac{\|r\|_{\infty}}{\|S_{y}^{-1}\|_{\infty}\|A_{x}^{-1}\|_{\infty}}
= \frac{\|r\|_{\infty}}{(1+|y|)\max\left\{\frac{1}{x}, \frac{1}{\sqrt{x}}\right\}}
\ge \frac{R-\delta}{2e^{\delta}\left(1+\frac{\delta}{2}\right)^{2}}
 = \frac{R'+\delta}{2}.
\]
It now follows from \eqref{eq:lem53} that $\|z\|_{\infty} \le \frac{\delta}{2}$. Hence
\[
\|z+S_{y}A_{x}r\|_{\infty} \ge \|S_{y}A_{x}r\|_{\infty} - \|z\|_{\infty} \ge \frac{R'}{2}.
\]

Summarizing our findings, we can conclude that $(p,q,r)\cdot(x,y,z) = (px, y+q\sqrt{x}, z +
S_{y}A_{x}r) \notin Q_{R'}$, i.e., \eqref{eq:claim11} was proven, and hence \eqref{eq:clam}.

To finish the proof, let $\epsilon > 0$. Since $\mathcal{SH}_{\psi}(\Lambda)$ is a frame,
Theorem \ref{theo:main}(i) (notice that its proof does only require results from Sections
\ref{sec:intro}--\ref{sec:B0}) implies that $\cD^{+}(\Lambda) < \infty$. By
Proposition \ref{prop:equivalent_D+finite_D-positive},
\[
M := \sup_{(x,y,z) \in \mathbb{S}} \#(\Lambda \cap {Q_1}(x,y,z)) < \infty.
\]
Hence, for all $(p,q,r) \in \mathbb{S}$,
\beq\label{eq:msup}
\sup_{(x,y,z) \in \mathbb{S}} \#((p,q,r)^{-1}\cdot\Lambda \cap {Q_1}(x,y,z))
= \sup_{(x,y,z) \in \mathbb{S}} \#\left(\Lambda \cap {Q_1}\left((p,q,r)^{-1}\cdot(x,y,z)\right)\right)
\le M < \infty.
\eeq

By Lemma \ref{lem:covering}, the set $\{{Q_1}(e^{j},ke^{-1/4}, e^{-1/2}m)\}_{j, k \in \Z, m \in \ZZ^2}$
is relatively separated and for each point $(x,y,z) \in (p,q,r)^{-1}\cdot\Lambda \setminus Q_{R}$ there exists
some $j_0,k_0,m_0$ such that $(x,y,z) \in Q_{1}(e^{j_0},k_0e^{-1/4}, e^{-1/2}m_0)$. Also,
Lemma \ref{lem:covering} implies that there does not exist any element $\{Q_{1}(e^{j},ke^{-1/4}, e^{-1/2}m)\}_{j, k \in \Z, m \in \ZZ^2}$
which intersects more than $24\left(e^{\frac{\delta}{2}}+\frac{1}{2}\right)\left(e^{\delta}+
\frac{1}{2}\right)\left(e^{\frac{3\delta}{4}}+\frac{1}{2}\right)$ of the other elements in this set.

In light of \eqref{eq:clam} we now define
\[
J:=\{(j,k,m) \in \Z\times\Z\times\Z^{2} : Q_{1}(e^{j},ke^{-1/4}, e^{-1/2}m) \cap Q_{R'} = \emptyset\}.
\]
Since $\psi, g \in \mathcal{B}_{0}$, by Theorem \ref{theo:b0}, we have $\mathcal{SH}_{\psi}g \in
W_{\mathbb{S}}(C,L^{1}) \subset W_{\mathbb{S}}(C,L^{2})$. Hence, if $R'$ is large enough,
\begin{equation}\label{eq:hap1}
\sum_{(j,k,m) \in J} \|\mathcal{SH}_{\psi}g \cdot \chi_{Q_{1}(e^{j},ke^{-1/4}, e^{-1/2}m)}\|_{\infty}^{2} < \frac{\epsilon}{M}.
\end{equation}

Finally, by \eqref{eq:msup} and \eqref{eq:hap1}, we conclude that
\[
\sum_{(x,y,z) \in(p,q,r)^{-1}\Lambda\setminus Q_{R}} \left|\mathcal{SH}_{\psi}g(x,y,z)\right|^{2}
\le M\sum_{(j,k,m) \in J} \|\mathcal{SH}_{\psi}g \cdot \chi_{Q_{1}(e^{j},ke^{-1/4}, e^{-1/2}m)}\|_{\infty}^{2}
< \epsilon.
\]
This proves the lemma.

\subsubsection{Proof of Theorem \ref{theo:comparison}}
\label{subsubsec:comparison}

Since $\{ \sigma(a,s,t)\phi: (a,s,t) \in \Delta\}$ is a frame with frame bounds $0<A\le B< \infty$,
\cite[Lem. 3]{Gro08} together with the definition of the space $V$ implies that
{\allowdisplaybreaks
\begin{eqnarray}\nonumber
\lefteqn{\mbox{\rm tr}[T]}\\ \nonumber
& \ge & \frac{1}{B} \sum_{(a,s,t) \in \Delta} \langle T(\sigma(a,s,t) \phi), \sigma(a,s,t) \phi \rangle\\ \nonumber
& \ge & \frac{1}{B} \sum_{(a,s,t) \in (p,q,r)Q_{h}\cap \Delta} \langle T(\sigma(a,s,t) \phi), \sigma(a,s,t) \phi \rangle\\ \nonumber
& = & \frac{1}{B} \sum_{(a,s,t) \in (p,q,r)Q_{h}\cap \Delta} \langle  P_{W}P_{V}(\sigma(a,s,t) \phi), P_{V}\sigma(a,s,t)\phi \rangle\\ \label{eq:comp1}
& = & \frac{1}{B} \sum_{(a,s,t) \in (p,q,r)Q_{h}\cap \Delta} \left[ \langle \sigma(a,s,t)\phi, \sigma(a,s,t)\phi \rangle
- \langle (P_{W}-I)(\sigma(a,s,t) \phi), \sigma(a,s,t) \phi \rangle \right].
\end{eqnarray}
}
By Cauchy-Schwarz inequality, \eqref{eq:compare1}, and the Weak HAP, we estimate the second term in \eqref{eq:comp1} by
\begin{eqnarray*}
\lefteqn{\left|\langle (P_{W}-I)(\sigma(a,s,t) \phi), \sigma(a,s,t) \phi \rangle\right|}\\
& \le & \| (P_{W}-I) \sigma(a,s,t)\phi \|_{2} \cdot \|\sigma(a,s,t)\phi\|_{2}\\
& = & \mbox{\rm dist}\left(\sigma(a,s,t)\phi, W(R+he^{\frac{R}{2}}+Rhe^{\frac{R}{4}}, (p,q,r))\right) \cdot \|\phi\|_{2}\\
& \le & \mbox{\rm dist}\left(\sigma(a,s,t)\phi, W\left(R, (a,s,t)\right)\right) \cdot \|\phi\|_{2}\\
& \le & \epsilon \|\phi\|_{2}\\
& \le & \epsilon C.
\end{eqnarray*}
This yields a lower bound for the trace
\beq \label{eq:comp2}
\mbox{\rm tr}[T]
\ge \frac{1}{B} \sum_{(a,s,t) \in (p,q,r)Q_{}h\cap \Delta} C(C-\epsilon)
= \frac{C(C-\epsilon)}{B} \# (Q_{h}(p,q,r) \cap \Delta).
\eeq

Next we aim for an upper bound for $\mbox{\rm tr}[T]$. Since $T$ is a product of the orthogonal
projections, its eigenvalues are between $0$ and $1$, and hence
\beq \label{eq:comp3}
\mbox{\rm tr}[T]
\le \mbox{\rm rank}(T)
\le \dim (W(R+he^{\frac{R}{2}}+Rhe^{\frac{R}{4}}, (p,q,r)))
\le \# \left((p,q,r)Q_{R+he^{\frac{R}{2}}+Rhe^{\frac{R}{4}}} \cap \Lambda \right).
\eeq

Combining \eqref{eq:comp2} and \eqref{eq:comp3},
\begin{eqnarray*}
\frac{C(C-\epsilon)}{B} \cdot \# (Q_{h}(p,q,r) \cap \Delta)
& \le & \# \left( Q_{R+he^{\frac{R}{2}}+Rhe^{\frac{R}{4}}}(p,q,r) \cap \Lambda \right)\\
& \le & \frac{\#\left( Q_{R+he^{\frac{R}{2}}+Rhe^{\frac{R}{4}}}(p,q,r) \cap \Lambda \right)}
{(R+he^{\frac{R}{2}}+Rhe^{\frac{R}{4}})^{4}}\cdot \frac{(R+he^{\frac{R}{2}}+Rhe^{\frac{R}{4}})^{4}}{h^{4}}.
\end{eqnarray*}

Taking the infimum or supremum over all points $(p,q,r) \in \mathbb{S}$, and then the $\liminf$ or $\limsup$ as $h \to
\infty$, we obtain
\[
\frac{C(C-\epsilon)}{B} \cD^{-}(\Delta) \le \cD^{-}(\Lambda) (e^{\frac{R}{2}}+Re^{\frac{R}{4}})^{4}
\quad \mbox{and} \quad
\frac{C(C-\epsilon)}{B} \cD^{+}(\Delta) \le \cD^{+}(\Lambda) (e^{\frac{R}{2}}+Re^{\frac{R}{4}})^{4},
\]
respectively. The claim is proved.



\begin{thebibliography}{99}

\bibitem{BCHL06}
R. Balan, P. G. Casazza, C. Heil, and Z. Landau,
{\em Density, overcompleteness, and localization of frames, I. Theory},
J. Fourier Anal. Appl. {\bf 12} (2006), 105--143.


\bibitem{CKS06}
W. Czaja, G. Kutyniok, and D. Speegle,
{\em The geometry of sets of parameters of wave packet frames},
Appl. Comput. Harmon. Anal. {\bf 20} (2006), 108--125.

\bibitem{DKMSST08} S. Dahlke, G. Kutyniok, P. Maass, C. Sagiv, H.-G. Stark, and G. Teschke,
\emph{The Uncertainty Principle Associated with the Continuous Shearlet Transform},
Int. J. Wavelets Multiresolut. Inf. Process \textbf{6} (2008), 157--181.

\bibitem{DKST09} S. Dahlke, G. Kutyniok, G. Steidl, and G. Teschke,
\emph{Shearlet Coorbit Spaces and associated Banach Frames},
Appl.Comput. Harmon. Anal. \textbf{27} (2009), 195--214.

\bibitem{ELL08}
G. Easley, D. Labate, and W. Lim,
{\em Sparse Directional Image Representations using the Discrete Shearlet Transform},
Appl. Comput. Harmon. Anal. {\bf 25} (2008), 25-–46.

\bibitem{Fei80}
H. G. Feichtinger,
{\em Banach convolution algebras of Wiener type},
Functions, Series, Operators, Proc. Conf. Budapest {\bf 38}, Colloq. Math. Soc. J\'{a}nos Bolyai (1980), 509--524.

\bibitem{GLWW03}
P. Gressman, D. Labate, G. Weiss, and E.N. Wilson,
{\em Affine, quasi-affine and co-affine wavelets},
Beyond Wavelets, Studies in Computational Mathematics, Elsevier, {\bf 10} (2003), 215--224.

\bibitem{Gro08}
K. Gr\"ochenig,
\emph{The Homogeneous Approximation Property and the Comparison Theorem for Coherent Frames},
Sampl. Theory Signal Image Process.  {\bf 3} (2008), 311--319.

\bibitem{GKS08}
K. Gr\"ochenig, G. Kutyniok, and K. Seip,
{\em Landau's necessary density conditions for LCA groups},
J. Funct. Anal. {\bf 255} (2008), 1831--1850.

\bibitem{Gro09a}
P. Grohs,
{\em Continuous Shearlet frames and resolution of the Wavefront Set},
Geometry preprint 2009/05, TU Graz, September 2009.

\bibitem{Gro09b}
P. Grohs,
{\em Continuous shearlet tight frames},
Geometry preprint 2009/07, TU Graz, December 2009.

\bibitem{GKL06}
K. Guo, G. Kutyniok, and D. Labate,
\emph{Sparse Multidimensional Representations using Anisotropic Dilation and Shear Operators},
Wavelets and Splines (Athens, GA, 2005), Nashboro Press, Nashville, TN (2006), 189--201.

\bibitem{GL07}
K. Guo and D. Labate,
\emph{Optimally Sparse Multidimensional Representation using Shearlets},
SIAM J. Math Anal. \textbf{39} (2007), 298--318.

\bibitem{Hei03}
C. Heil,
{\em An introduction to weighted Wiener amalgams},
Proc. International Conference on Wavelets and their Applications (Chennai, January 2002),
Allied Publishers, New Delhi (2003), 183--216.

\bibitem{Hei07}
C. Heil,
{\em History and evolution of the Density Theorem for Gabor frames},
J. Fourier Anal. Appl. {\bf 13} (2007), 113--166.

\bibitem{HK03}
C. Heil and G. Kutyniok,
\emph{Density of weighted wavelet frames},
J. Geom. Anal. {\bf 13} (2003), 479--493.

\bibitem{HK07}
C. Heil and G. Kutyniok,
{\em The Homogeneous Approximation Property for Wavelet Frames},
J. Approx. Theory {\bf 147} (2007), 28--46.

\bibitem{HLWW04}
E. Hernandez, D. Labate, G. Weiss and E. Wilson,
{\em Oversampling, quasi-affine frames and wave packets},
Appl. Comput. Harmon. Anal.  {\bf 16} (2004), 111--147.

\bibitem{KKL10} P. Kittipoom, G. Kutyniok, and W.-Q Lim,
\emph{Construction of Compactly Supported Shearlets},
preprint.

\bibitem{Kut07a}
G. Kutyniok,
\emph{Affine Density in Wavelet Analysis}, Springer, 2007.

\bibitem{Kut07b}
G. Kutyniok,
\emph{Affine density, frame bounds, and the admissibility condition for wavelet frames},
Constr. Approx. \textbf{25} (2007), 239-253.

\bibitem{KL07}
G. Kutyniok and D. Labate,
{\em Construction of Regular and Irregular Shearlet Frames},
J. Wavelet Theory and Appl. {\bf 1} (2007), 1--10.

\bibitem{KL09}
G. Kutyniok and D. Labate,
\emph{Resolution of the Wavefront Set using Continuous Shearlets},
Trans. Amer. Math. Soc. \textbf{ 361} (2009), 2719--2754.

\bibitem{KL10}
G. Kutyniok and W. Lim,
\emph{Compactly Supported Shearlets are Optimally Sparse},
preprint.

\bibitem{DKS08}
G. Kutyniok, M. Shahram, and D. L. Donoho,
{\em Development of a Digital Shearlet Transform Based on Pseudo-Polar FFT},
in Wavelets XIII (San Diego, CA, 2009), D. Van De Ville, V. K. Goyal und M. Papadakis, eds.,
74460B-1 - 74460B-13, SPIE Proc. {\bf 7446}, SPIE, Bellingham, WA, 2009.

\bibitem{Lan93}
H. Landau,
{\em On the density of phase-space expansions},
IEEE Trans. Inform. Th. {\bf 39} (1993), 1152--1156.

\bibitem{Lim09}
W. Lim,
{\em Discrete Shearlet Transform: New Multiscale Directional Image Representation},
Proc. SAMPTA'09, Marseille 2009.

\bibitem{RS95} J. Ramanathan and T. Steger,
\emph{Incompleteness of sparse coherent states},
Appl. Comput. Harmon. Anal. \textbf{2} (1995), 148--153.

\bibitem{SZ03}
W. Sun and X. Zhou,
\emph{Density and stability of wavelet frames},
Appl. Comput. Harmon. Anal. \textbf{15} (2003), 117--133.

\bibitem{SZ04}
W. Sun and X. Zhou,
\emph{Density of irregular wavelet frames},
Proc. Amer. Math. Soc. \textbf{132} (2004), 2377-2387.

 \end{thebibliography}
\end{document}